\documentclass{article}
\usepackage{amsthm,amssymb,amsmath, amsfonts}
\usepackage[normalem]{ulem}

\usepackage{authblk}
\usepackage{enumerate}
\usepackage{hyperref}
\usepackage{mathabx}
\usepackage{blkarray}
\usepackage{tikz} 
\usetikzlibrary{backgrounds}
\usetikzlibrary{shapes}
\usetikzlibrary{decorations.markings}
\usepackage{subfig}
\usepackage{todonotes}
\usetikzlibrary{snakes}
\hypersetup{
    pdftitle=   {Erd\H{o}s-P\'osa property of minor-models with prescribed vertex sets},
   pdfauthor=  {O-joung Kwon and D\'aniel Marx}
}
\newcommand\abs[1]{\lvert #1\rvert}
\newcommand\multiabs[1]{\lvert\lvert #1\rvert\rvert}

\usepackage{amsmath}
\newtheorem{theorem}{Theorem}[section]
\newtheorem{lemma}[theorem]{Lemma}

\newtheorem{proposition}[theorem]{Proposition}

\newtheorem*{thmmain}{Theorem \ref{thm:main}}
\newtheorem{claim}{Claim}

\theoremstyle{remark}

\newenvironment{clproof}{\begin{list}{}{%
              \setlength{\leftmargin}{5mm}%
              } \item {\it Proof.} }{\hfill$\lozenge$\end{list}\medskip}

\theoremstyle{definition}

\newcommand\tw{\operatorname{tw}}

\newcommand\cC{\mathcal{C}}
\newcommand\cF{\mathcal{F}}
\newcommand\cH{\mathcal{H}}
\newcommand\cG{\mathcal{G}}
\newcommand\cR{\mathcal{R}}
\newcommand\cP{\mathcal{P}}
\newcommand\cB{\mathcal{B}}

\newcommand\cS{\mathcal{S}}
\newcommand\cZ{\mathcal{Z}}
\newcommand\cY{\mathcal{Y}}

\newcommand\cc{\operatorname{cc}}

\newcommand{\EP}{Erd\H{o}s-P\'osa}

\begin{document}
\title{Erd\H{o}s-P\'osa property of minor-models with prescribed vertex sets}

\author[1]{O-joung Kwon\thanks{Supported by the National Research Foundation of Korea (NRF) grant funded by the Ministry of Education (No. NRF-2018R1D1A1B07050294).}}
\author[2]{D\'aniel Marx\thanks{Supported by ERC Consolidator Grant SYSTEMATICGRAPH (No. 725978).}}
\affil[1]{Department of Mathematics, Incheon National University, Incheon, South Korea.}
\affil[2]{Institute for Computer Science and Control, Hungarian Academy of Sciences, Budapest, Hungary}

\date{\today}

\maketitle

\begin{abstract}
A minor-model of a graph $H$ in a graph $G$ is a subgraph of $G$ that can be contracted to $H$. 
 We prove that for a positive integer $\ell$ and a non-empty planar graph $H$ with at least $\ell-1$ connected components, 
 there exists a function $f_{H, \ell}:\mathbb{N}\rightarrow \mathbb{R}$ satisfying the property that 
  every graph $G$ with a family of vertex subsets $Z_1, \ldots, Z_m$ contains either $k$ pairwise vertex-disjoint minor-models of $H$ each intersecting at least $\ell$ sets among prescribed vertex sets, 
 or a vertex subset of size at most $f_{H, \ell}(k)$ that meets all such minor-models of $H$. 
 This function $f_{H, \ell}$ is independent with the number $m$ of given sets, and thus, our result generalizes Mader's $\cS$-path Theorem, by applying $\ell=2$ and $H$ to be the one-vertex graph.
 We prove that such a function $f_{H, \ell}$ does not exist if $H$ consists of at most $\ell-2$ connected components. 
\end{abstract}

\section{Introduction}
A class $\mathcal{C}$ of graphs is said to have the \emph{Erd\H{o}s-P\'osa property} if there exists a function $f$ satisfying the following property:
for every graph $G$ and a positive integer $k$, either
$G$ contains either $k$ pairwise vertex-disjoint subgraphs each isomorphic to a graph in $\mathcal{C}$ or
a vertex set $T$ of size at most $f(k)$ such that $G-T$ has no subgraph isomorphic to a graph in $\mathcal{C}$.
Erd\H{o}s and P\'osa~\cite{ErdosP1965} showed that the class of all cycles has the Erd\H{o}s-P\'osa property.
Later, several variations of cycles having the Erd\H{o}s-P\'osa property have been investigated; for instance, directed cycles~\cite{ReedRST1996}, long cycles~\cite{FioriniH2014, MoussetNSW2016, RobertsonS1986}, cycles intersecting a prescribed vertex set~\cite{KakimuraKM2011,PontecorviW2012, HuyneJW2017}, and holes~\cite{KimK2018}. We refer to a survey of the Erd\H{o}s-P\'osa property by Raymond and Thilikos~\cite{RaymondT2017} for more examples.

This property has been extended to a class of graphs that contains some fixed graph as a minor.
A \emph{minor-model function} of a graph $H$ in a graph $G$ is a function $\eta$ with the domain $V(H)\cup E(H)$, where
\begin{itemize}
\item for every $v\in V(H)$, $\eta(v)$ is a non-empty connected subgraph of $G$, all pairwise vertex-disjoint
\item for every edge $e$ of $H$, $\eta(e)$ is an edge of $G$, all distinct
\item for every edge $e=uv$ of $H$, if $u\neq v$ then $\eta(e)$ has one end in $V(\eta(u))$ and
     the other in $V(\eta(v))$; and if $u = v$, then $\eta(e)$ is an edge of $G$ with all ends in $V(\eta(v))$.
\end{itemize}
We call the image of such a function an \emph{$H$-minor-model} in $G$, or shortly an \emph{$H$-model} in $G$.
We remark that an $H$-model is not necessarily a minimal subgraph that admits a minor-model function from $H$. For instance, when $H$ is the one-vertex graph, any connected subgraph is an $H$-model.
As an application of Grid Minor Theorem, Robertson and Seymour~\cite{RobertsonS1986} proved that  the class of all $H$-models has the Erd\H{o}s-P\'osa property if and only if $H$ is planar. 

Another remarkable result on packing and covering objects in a graph is Mader's $\cS$-path Theorem~\cite{Mader78}.
Mader's $\cS$-path Theorem states that for a family $\cS$ of vertex subsets of a graph $G$, $G$ contains either $k$ pairwise vertex-disjoint paths connecting two distinct sets in $\cS$, or a vertex subset of size at most $2k-2$ that meets all such paths. An interesting point of this theorem is that the number of sets in $\cS$ does not affect on the bound $2k-2$. A simplified proof was later given by Schrijver~\cite{Schrijver01}.

In this paper, we generalize Robertson and Seymour's theorem on the \EP\ property of $H$-models and Mader's $\cS$-path Theorem together.
For a graph $G$ and a multiset $\cZ$ of vertex subsets of $G$, 
the pair $(G, \cZ)$ is called \emph{a rooted graph}. 
For a positive integer $\ell$ and a family $\cZ$ of vertex subsets, an $H$-model $F$ is called an \emph{$(H, \cZ, \ell)$-model} if there are at least $\ell$ distinct sets $Z$ of $\cZ$ that contains a vertex of $F$.
A vertex set $S$  of $G$ is called an \emph{$(H, \cZ, \ell)$-deletion set} if $G- S$ has no $(H, \cZ, \ell)$-models.
For a graph $G$, we denote by $\cc(G)$ the number of connected components of $G$.

We completely classify when the class of $(H, \cZ, \ell)$-models has the \EP\ property.

\begin{theorem}\label{thm:main}
For a positive integer $\ell$ and a non-empty planar graph $H$ with $\cc(H)\ge \ell-1$,
there exists $f_{H, \ell}:\mathbb{N}\rightarrow \mathbb{R}$ satisfying the following property.
Let $(G, \cZ)$ be a rooted graph and $k$ be a positive integer.
Then $G$ contains either
$k$ pairwise vertex-disjoint $(H, \cZ, \ell)$-models in $G$, or
an $(H, \cZ, \ell)$-deletion set of size at most  $f_{H,\ell}(k)$.

If $\cc(H)\le \ell-2$, then such a function $f_{H, \ell}$ does not exist.
\end{theorem}
Together with the result of Robertson and Seymour~\cite{RobertsonS1986} on $H$-models, we can reformulate as follows. 
\begin{theorem}
The class of $(H, \cZ, \ell)$-models has the \EP\ property if and only if $H$ is planar and $\cc(H)\ge \ell-1$.
\end{theorem}

By setting $\cZ=\{V(G)\}$ and $\ell=1$, Theorem~\ref{thm:main} contains the \EP\ property of $H$-models.
We point out that the size of a deletion set in Theorem~\ref{thm:main} does not depend on the number of given prescribed vertex sets in $\cZ$. 
Thus it generalizes Mader's $\cS$-path Theorem, when $H$ is the one-vertex graph and $\ell=2$.

Here we give one example showing that the class of $(H, \cZ, \ell)$-models does not satisfy the \EP\ property when $H$ is planar, but consists of at most $\ell-2$ connected components.
Let $\ell=3$ and $H$ be a connected graph.
Let $G$ be an $(n\times n)$-grid with sufficiently large $n$, and let $Z_1, Z_2, Z_3$ be the set of all vertices in the first column, the first row, and the last column, respectively, except corner vertices.
See Figure~\ref{fig:counterex1}.
One can observe that there cannot exist two $H$-models in $G$ meeting all of $Z_1, Z_2, Z_3$ because $H$ is connected. On the other hand, we may increase the minimum size of an $(H, \{Z_1, Z_2, Z_3\}, \ell)$-deletion set as much as we can, by taking a sufficiently large grid.  In Section~\ref{sec:counterex},  we will generalize this argument for all pairs $H$ and $\ell$ with $\cc(H)\le \ell-2$.

We remark that Bruhn, Joos, and Schaudt~\cite{BJS2018} 
considered labelled-minors and provided a characterization for 2-connected labelled graphs $H$ where $H$-models intersecting a given set have the Erd\H{o}s-P\'osa property.
However, their minor-models are minimal subgraphs containing a graph $H$ as a minor. So the context is slightly different.

One of the main tools to prove Theorem~\ref{thm:main} is the rooted variant of Grid Minor Theorem (Theorem~\ref{thm:rootedgridminor}).
Note that just a large grid model may not contain many pairwise vertex-disjoint $(H, \cZ, \ell)$-models. 
We investigate a proper notion, called a \emph{$(\cZ,k,\ell)$-rooted grid model}, which contains many disjoint $(H, \cZ, \ell)$-models. Briefly, we show that every graph with a large grid model contains a $(\cZ, k, \ell)$-rooted grid model, or a separation of small order separating most of sets of $\cZ$ from the given grid model. 
Previously, Marx, Seymour, and Wollan~\cite{MarxSW2013} proved a similar result with one prescribed vertex set. 
The previous result was stated in terms of \emph{tangles}~\cite{RobertsonS91}, 
but we state in an elementary way without using tangles.
The advantage of our formulation is that we do not need to define a relative tangle at each time, in the induction step.

In Section~\ref{sec:pureep}, we introduce a pure $(H, \cZ, \ell)$-model that consists of a minimal set of connected components of an $(H, \cZ, \ell)$-model 
intersecting $\ell$ sets of $\cZ$. A formal definition can be found in the beginning of Section~\ref{sec:pureep}. Note that 
in a pure $(H, \cZ, \ell)$-model, each component has to intersect a set of $\cZ$, while in an ordinary $(H, \cZ, \ell)$-model, a component does not necessarily intersect a set of $\cZ$. Because of this property, when we have a separation $(A,B)$ where $B-V(A)$ contains few sets of $\cZ$ but contains a large grid, we may find an irrelevant vertex to reduce the instance. 
Starting from this observation, we will obtain the \EP\ property for pure $(H, \cZ, \ell)$-models (Theorem~\ref{thm:mainpure}).

In Section~\ref{sec:total}, we obtain the \EP\ property for $(H, \cZ, \ell$)-models, using the result for pure $(H, \cZ, \ell)$-models.
An observation is that every $(H, \cZ, \ell)$-model contains a pure $(H, \cZ, \ell)$-model, by taking components that essentially hits $\ell$ sets of $\cZ$.
So, any deletion set for pure $(H, \cZ, \ell)$-models hits all $(H, \cZ, \ell)$-models as well.
When we have a situation that a given graph has large grid model (with a proper separation) and $k$ pairwise vertex-disjoint pure $(H, \cZ, \ell)$-models, 
we complete these models to $(H, \cZ, \ell)$-models, by taking rest components from the large grid model.
This will complete the argument for \EP\ property of $(H, \cZ, \ell)$-models.

\begin{figure}[t]
\centerline{\includegraphics[scale=0.8]{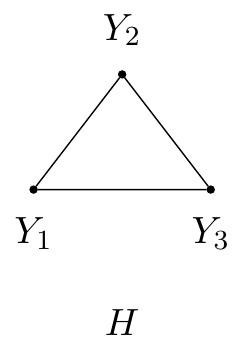}\quad\quad \includegraphics[scale=0.65]{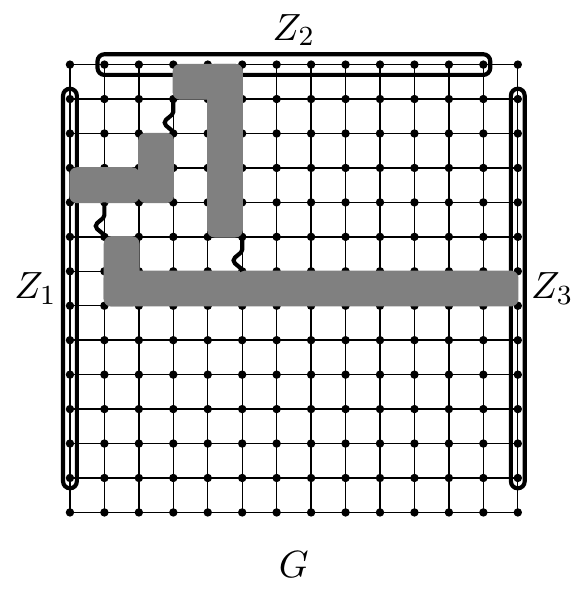}}
  \caption{There are no two pairwise vertex-disjoint $H$-models meeting all of $Z_1, Z_2, Z_3$.}
  \label{fig:counterex1}
\end{figure}

\section{Preliminaries}\label{sec:prelim}

All graphs in this paper are simple, finite and undirected.
For a graph $G$, we denote by $V(G)$ and $E(G)$ the vertex set and the edge set of $G$, respectively. 
For $S\subseteq V(G)$, we denote by $G[S]$ the subgraph of $G$ induced by $S$. 
For $S\subseteq V(G)$ and $v\in V(G)$, let $G- S$ be the graph obtained by removing all vertices in $S$, and let $G- v:=G- \{v\}$. 
For two graphs $G$ and $H$, we define $G\cap H$ as the graph on the vertex set $V(G)\cap V(H)$ and the edge set $E(G)\cap E(H)$, and define $G\cup H$ analogously.
A \emph{separation of order} $k$ in a graph $G$ is a pair $(A,B)$ of subgraphs of $G$ such that $A\cup B=G$, $E(A)\cap E(B)=\emptyset$, and $\abs{V(A\cap B)}=k$. 

For two disjoint vertex subsets $A$ and $B$ of a graph $G$, we say that $A$ is \emph{complete} to $B$ if for every vertex $v$ of $A$ and every vertex $w$ of $B$, $v$ is adjacent to $w$.

For a positive integer $n$, let $[n]:=\{1, \ldots, n\}$.
For two positive integers $m$ and $n$ with $m\le n$, let $[m,n]:=\{m, \ldots, n\}$.
For a set $A$, we denote by $2^A$ the set of all subsets of $A$.

For positive integers $g$ and $h$, 
the \emph{$(g\times h)$-grid} is  the graph on the vertex set $\{v_{i,j}:i\in [g], j\in [h]\}$
where $v_{i,j}$ and $v_{i',j'}$ are adjacent if and only if $\abs{i-i'}+\abs{j-j'}=1$.
For each $i\in [g]$, we call $\{v_{i,1}, \ldots, v_{i,h}\}$ \emph{the $i$-th row} of $G$, and define its columns similarly.
We denote by $\cG_g$ the $(g\times g)$-grid graph, and for a positive integer $\ell$, we denote by $\ell\cdot \cG_g$ the disjoint union of $\ell$ copies of $(g\times g)$-grid graphs.

A graph is \emph{planar} if it can be embedded on the plane without crossing edges.
A graph $H$ is a \emph{minor} of $G$ if $H$ can be obtained from a subgraph of $G$ by contracting edges.
It is well known that $H$ is a minor of $G$ if and only if $G$ has an $H$-model.

\paragraph{\bf Operations on multisets.}

A \emph{multiset} is a set with allowing repeatition of elements.
In particular, for a rooted graph $(G, \cZ)$, we consider $\cZ$ as a multiset.
For a multiset $\cZ$, we denote by $\multiabs{\cZ}$ the number of sets in $\cZ$, which counts elements with multiplicity, and does not count a possible empty set.
For example, $\multiabs{ \{A, B, B, C, C, \emptyset\} }=5$. Note that for an ordinary set $A$, we use $\abs{A}$ for the size of a set $A$ in a usual meaning.

Let $\cZ$ be a multiset of subsets of a set $A$.
For a subset $B$ of $A$, we define $\cZ|_{B}:= \{X\cap B: X\in \cZ\}$ and $\cZ\setminus B:=\{X\setminus B:X\in \cZ\}$. For convenience, when $(G, \cZ)$ is a rooted graph and $H$ is a subgraph of $G$, 
we write $\cZ|_{H}:=\cZ|_{V(H)}$ and $\cZ\setminus H:=\cZ\setminus V(H)$.

\paragraph{\bf Tree-width.}

A \emph{tree-decomposition} of a graph $G$ is 
a pair $(T,\cB)$ 
of a tree $T$
and a family $\cB=\{B_t\}_{t\in V(T)}$ of vertex sets $B_t\subseteq V(G)$,
called \emph{bags},
satisfying the following three conditions:
\begin{enumerate}
\item[(T1)] $V(G)=\bigcup_{t\in V(T)}B_t$.
\item[(T2)] For every edge $uv$ of $G$, there exists a vertex $t$ of $T$ such that $\{u, v\}\subseteq B_t$.
\item[(T3)] For $t_1$, $t_2$, and $t_3\in V(T)$, $B_{t_1}\cap B_{t_3}\subseteq B_{t_2}$ whenever $t_2$ is on the path from $t_1$ to $t_3$.
\end{enumerate}

The \emph{width} of a tree-decomposition $(T,\cB)$ is $\max\{ \abs{B_{t}}-1:t\in V(T)\}$. The \emph{tree-width} of $G$, denoted by $\tw(G)$, is the minimum width over all tree-decompositions of $G$. 
Robertson and Seymour~\cite{RobertsonS1986} showed that every graph of sufficiently large tree-width contains a big grid model.

\paragraph{\bf Grid minor-models.}

In the course of Graph Minors Project, Robertson and Seymour~\cite{RobertsonS1986} showed that every graph with sufficiently large tree-width contains a large grid as a minor. Our result is also based on this theorem.
\begin{theorem}[Grid Minor Theorem~\cite{RobertsonS1986}]\label{thm:gridtheorem}
For all $g\ge 1$, there exists $\kappa (g)\ge 1$ such that every graph of tree-width at least $\kappa (g)$ contains a minor isomorphic to $\cG_g$; in other words, it contains a $\cG_g$-model.
\end{theorem}
The original function $\kappa (g)$ due to Roberson and Seymour was a tower of exponential functions.
Chuznoy~\cite{Chuznoy2016} recently announced that $\kappa (g)$ can be taken to be $\mathcal{O}(g^{19}\operatorname{poly}\log g)$.
This polynomial bound gives a polynomial bound on the function $f$ in Theorem~\ref{thm:main}.
We also use a known result that 
if $H$ is planar and $n\ge 14\abs{V(H)}\ge 2\abs{V(H)}+4\abs{E(H)}$, then
$\cG_{n}$ contains an $H$-model~\cite{RobertsonST1994}.

Suppose that a graph $G$ contains a $\cG_{n}$-model.
When we remove $k$ vertices contained in the model, 
we can take a $\cG_{n-k}$-model in a special way.
The following lemma describes how we can take a smaller grid model.
\begin{lemma}\label{lem:restrictgrid}
Let $n>k\ge 0$.
Let $G$ be a graph having a $\cG_{n}$-model $H$ with a model function $\eta$, and let $S\subseteq V(G)$ with $\abs{S}=k$.
Then $G-S$ contains a $\cG_{n-k}$-model contained in $H$ such that it contains all rows and columns of $H$ that does not contain a vertex of $S$.
\end{lemma}
\begin{proof}
Let $i_1<i_2< \cdots <i_a$ be the set of indices of all rows of $H$ that do not contain a vertex of $S$.
Similarly, let $j_1<j_2< \cdots <j_b$ be the set of indices of all columns of $H$ that do not contain a vertex of $S$.
Clearly, $a\ge n-k$ and $b\ge n-k$.
Let $i_0=j_0=0$.

We define that for $x\in [a]$ and $y\in [b]$,
\begin{displaymath}
\alpha(v_{x,y}): = \left\{ \begin{array}{ll}
 H[\bigcup_{i_{x-1}< x'\le i_x} \eta(v_{x',j_y})\cup \bigcup_{j_{y-1}< y'\le j_y} \eta(v_{i_x, y'})] \\ 
 \qquad \qquad \qquad \textrm{if $x\in [a-1]$ and $y\in [b-1]$}\\
 H[\bigcup_{i_{x-1}< x'\le i_x} \eta(v_{x',j_y})\cup \bigcup_{j_{y-1}< y'\le n} \eta(v_{i_x, y'})] \\
 \qquad \qquad \qquad \textrm{if $x\in [a-1]$ and $y=b$}\\
 H[\bigcup_{i_{x-1}< x'\le n} \eta(v_{x',j_y})\cup \bigcup_{j_{y-1}< y'\le j_y} \eta(v_{i_x, y'})] \\ 
 \qquad \qquad \qquad \textrm{if $x\in a$ and $y\in [b-1]$}\\
 H[\bigcup_{i_{x-1}< x'\le n} \eta(v_{x',j_y})\cup \bigcup_{j_{y-1}< y'\le n} \eta(v_{i_x, y'})] \\ 
 \qquad \qquad \qquad \textrm{if $x\in a$ and $y\in b$.}\\
 \end{array} \right.
\end{displaymath}
These vertex-models with edges crossing between those models in $\cG_n$ induce a $\cG_{a, b}$-model such that it contains all rows and columns of $H$ that does not contain a vertex of $S$. By merging consecutive columns or rows if necessary, it also induces a $\cG_{n-k}$-model.
\end{proof}

\section{Finding a rooted grid model}\label{sec:colorfulrootedgrid}
In this section, we introduce a grid model with some additional conditions, called a \emph{$(\cZ,k,\ell)$-rooted grid model}.
The advantage of this notion is that for a planar graph $H$ with at least $\ell-1$ connected components, every $(\cZ,k,\ell)$-rooted grid model of sufficiently large order always contains many vertex-disjoint $(H, \cZ, \ell)$-models.

 Let $(G,\cZ=\{Z_i:i\in [m]\})$ be a rooted graph.
For a positive integer $k$, a vertex set $\{w_i:i\in [n]\}$ in $G$ is said to \emph{admit a $(\cZ, k)$-partition}
if there exist a partition $L_1, \ldots, L_x$ of $\{w_i: i\in [n]\}$ and an injection $\gamma:[x]\rightarrow[m]$ such that for each $i\in [x]$, $\abs{L_i}\le k$ and $L_i\subseteq Z_{\gamma (i)}$.
For positive integer $g,k, \ell$ with $g\ge k\ell$, we define \emph{a model function of a $(\cZ, k, \ell)$-rooted grid model of order $g$} as a model function $\eta$ of $\cG_g$ such that 
for each $i\in [k\ell]$, $V(\eta(v_{1,i}))$ contains a vertex $w_i$ and $\{w_i:i\in [k\ell]\}$ admits a $(\cZ, k)$-partition.
We call $w_1, \ldots, w_{k\ell}$ the \emph{root vertices} of the model.
The image of such a model $\eta$ is called a \emph{$(\cZ,k,\ell)$-rooted grid model of order $g$}.
A $(\emptyset, k, \ell)$-rooted grid model of order $g$ is just a model of $\cG_g$.

We prove the following.
\begin{theorem}\label{thm:rootedgridminor}
Let $g,k,\ell$, and $n$ be positive integers with $g\ge k\ell$ and $n\ge g(k^2 \ell^2+1)+k\ell$.
Every rooted graph $(G, \cZ)$ having a $\cG_n$-model contains either
\begin{enumerate}[(1)]
\item a separation $(A,B)$ of order less than $k(\ell-\multiabs{\cZ\setminus A})$ where $\multiabs{\cZ\setminus A}\le \ell-1$ and
$B-V(A)$ contains a $\cG_{n-\abs{V(A\cap B)}}$-model, or 
\item a $(\cZ, k, \ell)$-rooted grid model of order $g$.
\end{enumerate}
\end{theorem}

To prove Theorem~\ref{thm:rootedgridminor}, we prove a related variation of Menger's theorem.
For positive integers $k$ and $n$, a set of pairwise vertex-disjoint paths $P_1, \ldots, P_n$ from $\bigcup_{Z\in \cZ}Z$ to $Y$ is called \emph{a $(\cZ,Y,k)$-linkage of order $n$}
if the set of all end vertices of $P_1, \ldots, P_n$ in $\bigcup_{Z\in \cZ} Z$ admits a $(\cZ,k)$-partition. 

\begin{proposition}\label{prop:mengervar}
Let $k$ and $\ell$ be positive integers. Every rooted graph $(G,\cZ)$ with $Y\subseteq V(G)$
contains either
\begin{enumerate}[(1)]
\item a separation $(A,B)$ of order less than $k(\ell-\multiabs{\cZ\setminus A})$ such that
$Y\subseteq V(B)$ and $\multiabs{\cZ\setminus  A}\le \ell-1$,
or 
\item a $(\cZ,Y,k)$-linkage of order $k\ell$.
\end{enumerate}
\end{proposition}
\begin{proof}
Let $\cZ:=\{Z_i:i\in [m]\}$.
We obtain a graph $G'$ from $G$ as follows:
\begin{itemize}
\item for each $i\in [m]$, add a vertex set $W_i$ of size $k$  and make $W_i$ complete to $Z_i$.
\end{itemize}
It is not hard to observe that if there are $k\ell$ pairwise vertex-disjoint paths from $\bigcup_{i\in [m]}W_i$ to $Y$ in $G'$, then 
there is a $(\cZ,Y,k)$-linkage of order $k\ell$ in $G$.
Thus, by Menger's theorem, we may assume that there is a separation $(C,D)$ in $G'$ of order less than $k\ell$ with $\bigcup_{i\in [m]} W_i\subseteq V(C)$ and $Y\subseteq V(D)$.
We claim that there is a separation $(A,B)$ described in (1).

If $Z_j\setminus V(C)\neq\emptyset$ for some $j\in [m]$, then one vertex of $Z_j$ is contained in $V(D)\setminus V(C)$, and thus, $W_j$ should be contained in $V(D)$ as $W_j$ is complete to $Z_j$.   
Since $(C,D)$ has order less than $k\ell$, 
we have
\[\multiabs{\cZ\setminus C}\le \ell-1 .\]
Moreover, for every $j$ with $Z_j\setminus V(C)\neq \emptyset$, all vertices of $W_j$  are contained in $V(C\cap D)$, because $\bigcup_{i\in [m]} W_i\subseteq V(C)$.
Thus, if we take a restriction $(C\cap G,D\cap G)$ of $(C,D)$ on $G$, then at least $k\multiabs{\cZ\setminus C}$ vertices are removed from the separator $V(C\cap D)$. So, $(C\cap G, D\cap G)$ is a separation in $G$ of order less than $k(\ell-\multiabs{\cZ\setminus (C \cap G)}))$ such that
 $Y\subseteq V(D\cap G)$ and $\multiabs{\cZ\setminus (C\cap G)}=\multiabs{\cZ\setminus C}\le \ell-1$.
\end{proof}

\begin{proof}[Proof of Theorem~\ref{thm:rootedgridminor}]
Let $\cZ:=\{Z_i: i\in [m]\}$ and $H$ be the given $\cG_n$-model
with
a model function $\eta$.
We say that the image $\bigcup_{v\in R} \eta(v)$ of a column $R$ of $\cG_n$ is a column of $H$.
We will mark a set of columns of $H$ and use a constructive sequence of unmarked columns to construct the required grid model in (2).

\begin{claim}\label{claim:seqcolumn}
For $t\le \frac{n}{k\ell}$, 
$G$ contains a separation described in (1) or
a sequence $(\cP_1, \cR_1), \ldots, (\cP_{t}, \cR_{t})$ where
\begin{itemize}
\item $\cR_i$ is a set of $k\ell$ columns of $H$, and for $i\neq j$, $\cR_i$ and $\cR_j$ are disjoint,
\item $\cP_i$ is a $(\cZ,T,k)$-linkage of order $k\ell$ for some set $T$ of $k\ell$ vertices contained in pairwise distinct columns of $\cR_i$,
\item for every column $R\notin \bigcup_{i\in [t]} \cR_i$, none of the paths in $\bigcup_{i\in [t]}\cP_i$ meet $R$.  
\end{itemize}
\end{claim}
\begin{clproof}
We inductively find a sequence $(\cP_1, \cR_1), \ldots, (\cP_{t}, \cR_{t})$ if a separation described in (1) does not exist. Suppose that there is such a sequence $(\cP_1, \cR_1), \ldots, (\cP_{t-1}, \cR_{t-1})$.
Choose a set $\cR$ of $k\ell$ arbitrary columns not in $\bigcup_{i\in [\ell-1]}\cR_i$. Such a set of columns exists as $n\ge t k\ell$.
We choose $k\ell$ vertices from distinct columns of $\cR$, and say $T$. 
By Proposition~\ref{prop:mengervar}, 
$G$ contains a separation $(A,B)$ of order less than $k(\ell-\multiabs{\cZ\setminus A})$ where 
 $T\subseteq V(B)$ and $\multiabs{\cZ\setminus A}\le \ell-1$, or
a $(\cZ,T,k)$-linkage of order $k\ell$ in $G$.

Suppose the former separation exists.
By Lemma~\ref{lem:restrictgrid}, $G-V(A\cap B)$ contains a $\cG_{n-\abs{V(A\cap B)}}$-model $H'$, such that it contains all rows and columns of $H$ that does not contain a vertex of $A\cap B$. 
Note that $\abs{V(A\cap B)}<k\ell = \abs{T}$.
Since we have chosen $T$ from $k\ell$ distinct columns of $H$, $H$ contains a column that contains a vertex of $T$ but does not contain a vertex of $A\cap B$. Thus, this column is contained in $H'$. 
It implies that $H'$ is contained in $B- V(A)$.
As $\abs{V(A\cap B)}< k\ell$, it follows that $(A,B)$ is a separation described in (1), a contradiction.

Therefore, 
$G$ contains a $(\cZ,T,k)$-linkage $\cP$ of order $k\ell$. 
To guarantee the last condition of the claim, 
pick a path $P\in \cP$, and if possible, shorten it such that its new second end vertex is in a column not in $\bigcup_{i\in [t-1]}\cR_i$ and this column is different from where the second end vertices of the other paths in $\cP$ are. We perform such shortenings as long as possible. If it is not possible to shorten anymore, then 
we let $\cP_{t}=\cP$ and $\cR_{t}$ to be the set of columns where the second end vertices of the paths in $\cP$ are.
Now, it is clear that the paths in $\cP_{t}$ do not visit any column not in $\bigcup_{i\in [t]}\cR_i$. 
\end{clproof}

Let $(\cP_1, \cR_1), \ldots, (\cP_{k\ell}, \cR_{k\ell})$ be a sequence given by Claim~\ref{claim:seqcolumn}. Since $n\ge g(k^2\ell^2+1)+k\ell$, there exists $p$ with $k\ell\le p\le n-g$ such that
for every $1\le i\le g$, the $(p+i)$-th column is not in $\cR$.
Let 
\begin{align*}
X:=\bigcup_{i\in [k\ell+1,2k\ell]} \eta(v_{i, p+1}), \qquad
D:=\bigcup_{i\in [k\ell+1,n]}\bigcup_{j\in [g]} \eta(v_{i, p+j}).
\end{align*}
Let $G'$ be the graph obtained from $G$ by deleting $D\setminus X$ and contracting the set $\eta(v_{i, p+1})$ to a vertex $w_{i, p+1}$ for each $i\in [k\ell+1, 2k\ell]$.
Let $W:=\{w_{i, p+1}:i\in [k\ell+1, 2k\ell]\}$ and $\cZ':=\cZ|_{V(G)\setminus D}$.

\begin{claim}
There is no separation $(A,B)$ in $G'$ of order less than $k(\ell-\multiabs{\cZ'\setminus A})$ such that 
$W\subseteq V(B)$ and $\multiabs{\cZ'\setminus A}\le \ell-1$.
\end{claim}
\begin{clproof}
Assume that there is such a separation $(A,B)$ in $G'$. 
Since the sets of columns $\cR_1, \ldots, \cR_{k\ell}$ are disjoint and $\abs{V(A\cap B)}<k\ell$, 
there is an integer $1\le s\le k\ell$ such that $V(A\cap B)$ is disjoint from all of columns in $\cR_s$.
First suppose that all of columns in $\cR_s$ are contained in $B- V(A)$.
Since $\abs{V(A\cap B)}\le k(\ell-\multiabs{\cZ'\setminus A})-1$, 
among the paths in $\cP_s$, there are $k\multiabs{\cZ'\setminus A} +1$ paths fully contained in $B- V(A)$.
On the other hand, by the definition of $(\cZ,T,k)$-linkages, 
the set of all end vertices of paths in $\cP_s$ on the vertex set $\bigcup_{Z\in \cZ} Z$ admits a $(\cZ,k)$-partition, that is, 
\begin{itemize} 
\item there exist a partition $L_1, \ldots, L_x$ of the end vertices in $\bigcup_{i\in [m]} Z_i$ and an injection $\gamma:[x]\rightarrow[m]$ where for each $i\in [x]$,  $\abs{L_i}\le k$ and $L_i$ is contained in $Z_{\gamma(i)}$.
\end{itemize}
From the condition that each $L_i$ has size at most $k$, at least $\multiabs{\cZ'\setminus A} +1$ paths of $\cP_s$ are fully contained in $B- V(A)$ and have end vertices in pairwise distinct sets of $\{L_i: i\in [x]\}$. 
It means that $\multiabs{\cZ'\setminus A}\ge \multiabs{\cZ'\setminus A} +1$, which is a contradiction.
We conclude that there is a column $R$ in $\cR_s$ that is contained in $A- V(B)$.

We observe that there are $k\ell$ vertex-disjoint paths from $R$ to $W$ in $G'$.
If $R$ is the $i$-th column of $H$ for some $i\le p$, then we can simply use paths from $(k\ell+1)$-th, $\ldots$, $(2k\ell)$-th rows. 
Assume that $R$ is the $i$-th column where $i\ge p+1$. Then for each $t\in [k\ell]$, we construct a $t$-th path such that it starts in $\eta(v_{t,i})$, goes to $\eta(v_{t, p+t-k\ell})$ then goes to $\eta(v_{2k\ell+1-t, p+t-k\ell})$, and then terminates in $w_{2k\ell+1-t, p+1}$.  This is possible because $p\ge k\ell$.
One of these $k\ell$ paths is disjoint from $V(A\cap B)$, hence there is a vertex of $W$ in $V(A)\setminus V(B)$, contradicting the assumption that $W\subseteq V(B)$. 
\end{clproof}

Therefore, by Proposition~\ref{prop:mengervar}, 
$G'$ contains a $(\cZ',W,k)$-linkage $\cP$ of order $k\ell$.
Let $\eta'$ be a model of $\cG_g$ where $\eta'(v_{i,j})=\eta(v_{k\ell+i, p+j})$. This model can be extended by the paths in $\cP$, and it satisfies the conditions of the required model.
\end{proof}

We show that every sufficiently large $(\cZ,k,\ell)$-rooted grid model contains
$k$ pairwise vertex-disjoint $(H, \cZ, \ell)$-models. 

\begin{lemma}\label{lem:colorfulgrid}
Let $k, \ell$, and $h$ be positive integers. 
Every rooted graph $(G,\cZ)$ having a $(\cZ,k, \ell)$-rooted grid model of order $k\ell(h+2) +1$ 
contains $k$ pairwise vertex-disjoint $(\ell\cdot\cG_{h}, \cZ, \ell)$-models.
Moreover, if $\ell\ge 2$, then 
$G$ contains $k$ pairwise vertex-disjoint $((\ell-1)\cdot\cG_{h}, \cZ, \ell)$-models. 
\end{lemma}

The usefulness of the $(\cZ, k)$-partition is given in the next lemma. 
\begin{lemma}\label{lem:zkpartition}
Let $(G,\cZ=\{Z_1, \cdots, Z_m\})$ be a rooted graph. Every $(\cZ, k)$-partition of a vertex set $\{w_1, \ldots, w_{k\ell}\}$
admits a partition $I_1, \ldots, I_k$ of the index set $\{1,\ldots, k\ell\}$ such that for each $j\in [k]$, 
\begin{itemize}
\item $\abs{I_j}=\ell$, 
\item there is an injection $\beta_j:I_j\rightarrow[m]$ where for each $i\in I_j$, $w_i$ is contained in $Z_{\beta_j (i)}$, and
\item there are two integers $a_j, b_j\in I_j$ with $a_j<b_j$ where there is no integer $c$ in $\bigcup_{i\in [j,k]}I_i$ with $a_j<c<b_j$.
\end{itemize}
\end{lemma}

\begin{proof}
We inductively find such a partition $I_1, \ldots, I_k$ of $\{w_1, \ldots, w_{k\ell}\}$.
If $k=1$, then by the definition of a $(\cZ, k)$-partition, 
there is an injection $\gamma:\{1, \ldots, \ell\}\rightarrow[m]$ where for each $i\in [\ell]$, $w_i\in Z_{\gamma (i)}$.
Thus, $I_1=\{1,\ldots, \ell\}$ satisfies the property. Let us assume that $k\ge 2$, and
let $L_1, \ldots, L_x$ be a partition of $\{w_1, \ldots, w_{k\ell}\}$ and $\gamma:[x]\rightarrow[m]$ be an injection where for each $i\in [x]$, $\abs{L_i}\le k$ and $L_i$ is contained in $Z_{\gamma (i)}$.

Let us choose a vertex set $S$ of size $\ell$ from $\bigcup_{j\in [x]}L_j$ so that
\begin{itemize}
\item for each $j\in [x]$, $\abs{S\cap L_j}\le 1$,
\item if $\abs{L_i}=k$, then $S\cap L_i\neq \emptyset$,
\item there are two vertices $w_p$ and $w_{p+1}$ in $S$ with consecutive indices.
\end{itemize}
If there is a set $L_i$ with $\abs{L_i}=k$, then we first choose two consecutive vertices where one is contained in $L_i$ and the other is not in $L_i$, 
and then choose one vertex from each set $L$ of $\{L_1, \ldots, L_x\}$ with $\abs{L}=k$ that are not selected before.
If there is no set $L_i$ with $\abs{L_i}=k$, then we choose any two consecutive vertices that are contained in two distinct sets of $\cZ$.
Note that the number of selected vertices cannot be larger than $\ell$. 
If necessary, by putting some vertices from sets of $\{L_1, \ldots, L_x\}$ which are not selected before,
we can fill in $S$ so that $\abs{S}=\ell$, there are two vertices with consecutive indices, and 
there is an injection $\beta_1$ from the index set of $S$ to $[m]$ where for each $w_i\in S$, $w_i$ is contained in $Z_{\beta_1 (i)}$.

By taking a restriction, we can obtain a partition $L_1', \ldots, L_x'$ of $\{w_i:i\in [k\ell]\}\setminus S$ and an injection $\gamma':[x]\rightarrow[m]$ where $\abs{L_i'}\le k-1$ and $L_i'$ is contained in $Z_{\gamma'(i)}$ for each $i\in [x]$.
We give a new ordering $v'_1, \ldots, v'_{k(\ell-1)}$ of the vertex set $\{w_i:i\in [k\ell]\}\setminus S$ following the order in $\{w_i:i\in [k\ell]\}$.
By induction,
there exists a partition $I_2', \ldots, I_{k}'$ of $[(k-1)\ell]$ such that for each $j\in [2,k]$, 
\begin{itemize}
\item $\abs{I_j'}=\ell-1$ and there is an injection $\beta_j:I_j'\rightarrow[m]$ where for each $i\in I_j'$, $w_i'$ is contained in $Z_{\beta_j (i)}$, and
\item there are two integers $a_j, b_j\in I_j'$ with $a_j<b_j$ where there is no integer $c$ in $\bigcup_{j\le i\le k}I_i'$ with $a_j<c<b_j$.
\end{itemize}
It gives a corresponding partition $I_2, \ldots, I_k$ of the index set of  $\{w_i:i\in [k\ell]\}\setminus S$.
Let $I_1$ be the set of indices of vertices in $S$.
Then $\{I_i:i\in [k]\}$ is a required partition.
\end{proof}

\begin{figure}[t]
\centerline{\includegraphics[scale=0.8]{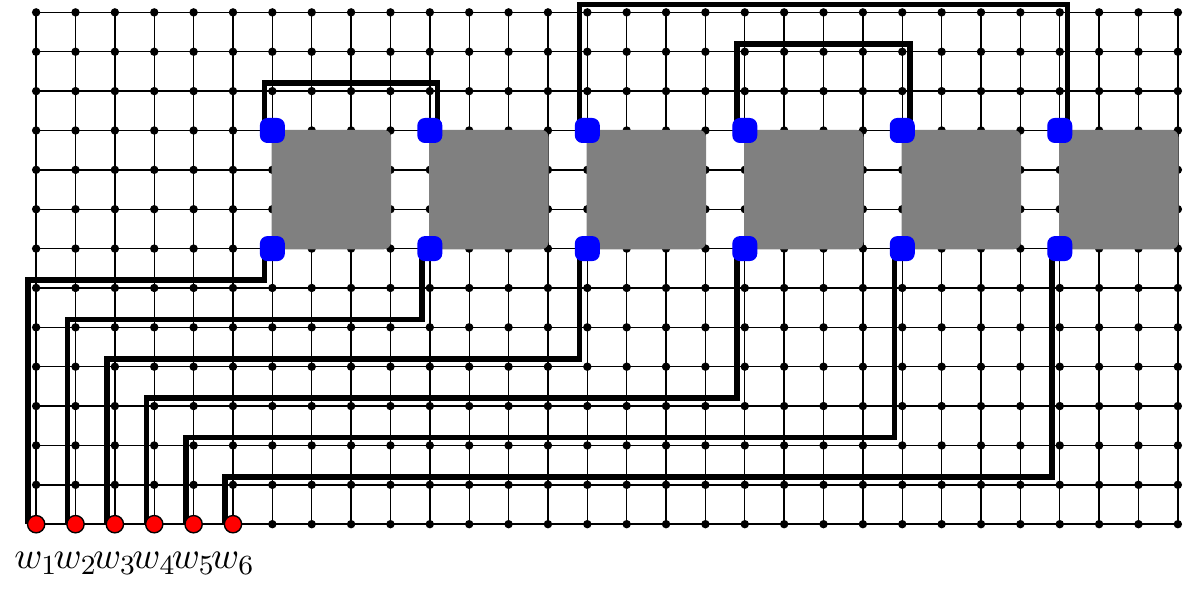}}
  \caption{This is a $(\{Z_1, Z_2, Z_3\}, 3, 2)$-rooted grid model where $Z_1=\{w_1, w_3, w_4\}$, $Z_2=\{w_2, w_6\}$, $Z_3=\{w_5\}$.
  We can obtain a partition of $\{1,\ldots, 6\}$ as $I_1=\{1,2\}$, $I_2=\{4,5\}$, $I_3=\{3,6\}$, as stated in Lemma~\ref{lem:zkpartition}.
  There are $3$ vertex-disjoint $(2\cdot\cG_{4}, \cZ, 2)$-models $G_1$, $G_2$, $G_3$ such that each $G_i$ consists of subgraphs containing vertices with indices in $I_i$.
  Also, there are $3$ vertex-disjoint  $(\cG_{4}, \cZ, 2)$-models using the three vertex-disjoint paths that connects two subgraphs in the same $G_i$.}
  \label{fig:rootedexpan}
\end{figure}

\begin{proof}[Proof of Lemma~\ref{lem:colorfulgrid}]
We first construct $k$ pairwise vertex-disjoint $(\ell\cdot\cG_{h}, \cZ, \ell)$-models.
Let $\eta$ be the model function of the $(\cZ,k, \ell)$-rooted grid model and let $w_1, \ldots, w_{k\ell}$ be the root vertices of the model where $w_i\in \eta(v_{i,1})$ for each $i\in [k\ell]$.
By Lemma~\ref{lem:zkpartition}, 
there exists a partition $I_1, \ldots, I_k$ of $[k\ell]$ such that for each $j\in [k]$, 
\begin{itemize}
\item $\abs{I_j}=\ell$ and there is an injection $\beta_j:I_j\rightarrow[m]$ where for each $i\in I_j$, $w_i$ is contained in $Z_{\beta_j (i)}$, and
\item there are two integers $a, b\in I_j$ with $a<b$ where there is no integer $c$ in $\bigcup_{i\in [j,k]}I_i$ with $a<c<b$.
\end{itemize}

For each $j\in [k\ell]$, 
 we define a subgraph $Q_j$ of $\cG_{k\ell(h+2)+1}$ as 
   \begin{align*}
Q_j:=(\cG_{k\ell(h+2)+1})[&\{v_{a,j}:a\in [1,k\ell+2-j]\}  \\ 
					\cup &\{v_{k\ell+2-j,b}:b\in [j,k\ell+1+h(j-1)]\} \\
                             \cup &\{v_{a,k\ell+1+h(j-1)}:a\in [k\ell+2-j,k\ell+2]\} \\
                               \cup &\{v_{a,b}: a\in [k\ell+2,k\ell+1+h], \\
                                & \qquad \quad b\in [k\ell+1+h(j-1),k\ell+hj]\}].
\end{align*}
Note that each $Q_j$ consists of a path from a vertex of $\{w_1, \ldots, w_{k\ell}\}$ to 
a $\cG_{h}$-model.
See Figure~\ref{fig:rootedexpan} for an illustration.
For each $j\in [k]$, let
\[G_j:=\eta\left( \bigcup_{i\in I_j} V(Q_i) \right).\]
It is not hard to observe that each $G_j$ is a $(\ell\cdot \cG_{h}, \cZ, \ell)$-model.
Thus, we obtain $k$ pairwise vertex-disjoint $(\ell\cdot \cG_{h}, \cZ, \ell)$-models.

Now, we prove the second statement. Let us assume that $\ell\ge 2$. We first construct the subgraphs $Q_j$ in the same way.
Then, for each $j\in [k]$, we add a path $P_j$ as follows:
$P_j$ starts in  
$v_{k\ell+1+h, k\ell+1+h(a_j-1)}$, 
goes to $v_{k\ell+1+h+i, k\ell+1+h(a_j-1)}$ 
then goes to $v_{k\ell+1+h+i, k\ell+1+h(b_j-1)}$, 
and then terminates in $v_{k\ell+1+h, k\ell+1+h(b_j-1)}$.
Note that $P_1, \ldots, P_k$ are pairwise vertex-disjoint in $\cG_{k\ell(h+2)+1}$ because of the second condition of the partition $I_1, \ldots, I_k$.
So, for each $j\in [k]$, the union of $G_j\cup \eta(V(P_j))$ is a $((\ell-1)\cdot \cG_{h}, \cZ, \ell)$-model, and
thus, $(\cZ,k, \ell)$-rooted grid model of order $k\ell(h+2)+1$ contains $k$ pairwise vertex-disjoint $((\ell-1)\cdot \cG_{h}, \cZ, \ell)$-models.
\end{proof}

\section{Erd\H{o}s-P\'osa property of pure $(H, \cZ, \ell)$-models}\label{sec:pureep}

Let $(G, \cZ)$ be a rooted graph and $H$ be a graph.
A subgraph $F$ of $G$ is called a  \emph{pure $(H, \cZ, \ell)$-model}
if there exist a subset $\cH=\{H_1, \ldots, H_t\}$ of the set of connected components of $H$  
and a function $\alpha:\{1, \ldots, t\}\rightarrow 2^{\cZ}\setminus \{\emptyset\}$ such that
\begin{itemize}
\item $F$ is the image of a model $\eta$ of $H_1\cup \cdots \cup H_t$ in $G$, 
\item for $Z\in \alpha(i)$, $\eta(V(H_i))\cap Z\neq \emptyset$, 
\item $\alpha(1), \ldots, \alpha(t)$ are pairwise disjoint and $\multiabs{\bigcup_{i\in [t]} \alpha(i)}=\ell$.
\end{itemize}
A pure $(H, \cZ, \ell)$-model can be seen as a minimal set of connected components of an $(H, \cZ, \ell)$-model for intersecting at least $\ell$ sets of $\cZ$. Because of the second and third conditions, 
every pure $(H, \cZ, \ell)$-model consists of the image of at most $\ell$ connected components of $H$.
A pure $(H,\cZ,\ell)$-model may be an $(H, \cZ, \ell)$-model itself when $H$ consists of $\ell-1$ or $\ell$ connected components.

In this section, we establish the \EP\ property for pure $(H, \cZ, \ell)$-models.
The reason for using the pure $(H, \cZ, \ell)$-models is that 
a procedure to find an irrelevant vertex in a graph of large tree-width works well for pure $(H, \cZ, \ell)$-models.
Later, we will argue how the problem for $(H, \cZ, \ell)$-models can be reduced to the problem for pure $(H, \cZ, \ell)$-models.
Generally, it is possible that $G$ has more than $k$ pairwise vertex-disjoint pure $(H, \cZ, \ell)$-models, even though it has no $k$ vertex-disjoint $(H, \cZ, \ell)$-models. So the relation is not very direct.

A vertex set $S$ of $G$ is called a \emph{pure $(H, \cZ, \ell)$-deletion set} if $G- S$ has no pure $(H, \cZ, \ell)$-models.

\begin{theorem}\label{thm:mainpure}
For a positive integer $\ell$ and a non-empty planar graph $H$ with $\cc(H)\ge \ell-1$, 
there exists $f^1_{H, \ell}:\mathbb{N}\rightarrow \mathbb{R}$ satisfying the following property.
Let $(G,\cZ)$ be a rooted graph and $k$ be a positive integer. 
Then $G$ contains either
$k$ pairwise vertex-disjoint pure $(H, \cZ, \ell)$-models, or
a pure $(H, \cZ, \ell)$-deletion set of size at most $f^1_{H, \ell}(k)$.
\end{theorem}

\subsection{Erd\H{o}s-P\'osa property of $(H, \cZ, \ell)$-models in graphs of bounded tree-width}
When the underlying graph has bounded tree-width, 
every class of graphs with at most $t$ connected components for some fixed $t$ has the \EP\ property.
It follows from the \EP\ property of subgraphs in a tree consisting of at most $d$ connected components for some fixed $d$, which was proved by Gy{\'a}rf{\'a}s and Lehel~\cite{GyarfasL70}. Later, the bound on a hitting set was improved by Berger~\cite{Berger2004}.

For a tree $T$ and a positive integer $d$, a subgraph of $T$ is called a \emph{$d$-subtree} if it consists of at most $d$ connected components.
\begin{theorem}[Berger~\cite{Berger2004}]\label{thm:dsubtree}
Let $T$ be a tree and let $k$ and $d$ be positive integers. Let $\cF$ be a set of $d$-subtrees of $T$.
Then $T$ contains either
$k$ pairwise vertex-disjoint subgraphs in $\cF$, or
a vertex set $S$ of size at most $(d^2-d+1)(k-1)$ such that $T- S$ has no subgraphs in $\cF$.
\end{theorem}

\begin{proposition}\label{prop:eponboundedtw}
Let  $k,h, \ell$, and $w$ be positive integers and let $H$ be a graph with $h$ vertices.
Let $(G,\cZ)$ be a rooted graph with $\tw(G)\le w$.
Then $G$ contains either 
$k$ pairwise vertex-disjoint $(H, \cZ, \ell)$-models, or
an $(H, \cZ, \ell)$-deletion set of size at most $(w-1)(h^2-h+1)(k-1)$.
Furthermore, the same statement holds for pure $(H, \cZ, \ell)$-models.
\end{proposition}
\begin{proof}
We first prove for usual $(H, \cZ, \ell)$-models. 
Let $\cH$ be the class of all $(H, \cZ, \ell)$-models.
Let $(T, \cB:=\{B_t\}_{t\in V(T)})$ be a tree-decomposition of $G$ of width at most $w$.
For $v\in V(G)$, let $\cP(v)$ denote the set of the vertices $t$ in $T$ such that $B_t$ contains $v$.
From the definition of tree-decompositions, 
for each $v\in V(G)$, $T[\cP(v)]$ is connected and 
for every edge $xy$ in $G$, there exists $B\in \cB$ containing both $x$ and $y$.
It implies that for a connected subgraph $F$ of $G$, 
$T[\bigcup_{v\in V(F)} \cP(v)]$ is connected.

Let $\cF$ be the family of sets $\bigcup_{x\in V(F)} \cP(x)$ for all $F\in \cH$. Observe that for $F_1, F_2\in \cH$ where $\bigcup_{x\in V(F_1)} \cP(x)$ and $\bigcup_{x\in V(F_2)} \cP(x)$ are disjoint, $F_1$ and $F_2$ are vertex-disjoint.
For a set $S\in \cF$, $T[S]$ consists of at most $h$ connected components, which means that it is a $h$-subtree.
So, if $\cF$ contains $k$ pairwise disjoint sets, then clearly, $\cH$ contains $k$ pairwise vertex-disjoint subgraphs of $\cH$ in $G$. 
Thus, we may assume that there are no $k$ pairwise disjoint subsets in $\cF$.
Then by Theorem~\ref{thm:dsubtree}, 
$T$ has a vertex set $W$  of size at most $(h^2-h+1)(k-1)$ such that $W$ meets all sets in $\cF$.
It implies that $\bigcup_{t\in W} B_t$ has at most $(w-1)(h^2-h+1)(k-1)$ vertices and it meets all  subgraphs in $\cH$.

The same argument holds for pure $(H, \cZ, \ell)$-models.
\end{proof}

\subsection{Reduction to graphs of bounded tree-width}

Let $(G, \cZ)$ be a rooted graph and $H$ be a graph.
Let $\tau^*_{H}(G,\cZ, \ell)$ be the minimum size of a pure $(H, \cZ, \ell)$-deletion set of $G$.
A vertex $v$ of $G$ is called \emph{irrelevant for pure $(H, \cZ, \ell)$-models}
if $\tau^*_{H}(G,\cZ, \ell)=\tau^*_{H}(G- v, \cZ, \ell)$. If there is no confuse from the context, then we shortly say that $v$ is an irrelevant vertex.

In this subsection, we argue that for two simpler cases when $G$ has large tree-width, one can obtain an irrelevant vertex. This will help to show base cases of Theorem~\ref{thm:mainpure}.

\begin{proposition}\label{prop:irrelevant}
Let $g,h$, and $\ell$ be positive integers and $x$ be a non-negative integer such that $g\ge (x^2+14hx+2x+1)(x^2+1)$. 
Let $(G,\cZ)$ be a rooted graph, 
and $H$ be a planar graph with $h$ vertices and $\cc(H)\ge \ell-1$,
and $(A,B)$ be a separation in $G$ of order at most $x$
such that $\multiabs{\cZ\setminus A}=0$ and $B-V(A)$ contains a $\cG_{g}$-model.
Then there exists an irrelevant vertex for pure $(H, \cZ, \ell)$-models.
\end{proposition}
\begin{proof}
We proceed it by induction on $x$.
If $x=0$, then a pure $(H, \cZ, \ell)$-model cannot have a vertex of $B- V(A)$ because every connected component of a pure $(H, \cZ, \ell)$-model contains at least one vertex of $\bigcup_{Z\in \cZ} Z$. 
Thus, every vertex in $B- V(A)$ is irrelevant.
We may assume that $x\ge 1$.

Let $W:=V(A\cap B)$ and $x':=x^2+14hx+2x$.
Note that $g\ge x'(x^2+1)+x^2$.
By applying Theorem~\ref{thm:rootedgridminor} to $B$ (with $(\cZ, g, k, \ell)\leftarrow (\{W\}, x', \abs{W}, 1)$), 
we deduce that $B$ contains either
\begin{enumerate}[(1)]
\item a separation $(A',B')$ of order less than $\abs{W}$ in $B$ such that
$W\subseteq V(A')$ and $B'- V(A')$ contains a $\cG_{g-\abs{V(A'\cap B')}}$-model, or
\item a $(\{W\}, \abs{W}, 1)$-rooted grid model of order $x'$.
\end{enumerate}

Suppose that there is a separation $(A',B')$ described in (1).
Since $W\subseteq V(A')$,$((A- E(A\cap B))\cup A', B')$ is a separation in $G$ of order at most $\abs{W}-1\le x-1$ such that
$\multiabs{\cZ\setminus (A\cup A')}=0$ and $B'- V(A\cup A')$ contains a $\cG_{g-\abs{V(A'\cap B')}}$-model. Thus $B'-V(A\cup A')$ contains a $\cG_{g-(x-1)}$-model.
Note that 
\begin{align*}
g-(x-1)&\ge (x^2+14hx+2x+1)((x-1)^2+1) \\
		&\ge ((x-1)^2+14h(x-1)+2(x-1)+1)((x-1)^2+1).
\end{align*}
Thus, by the induction hypothesis, $G$ contains an irrelevant vertex.

So, we may assume that $B$ contains a $(\{W\}, \abs{W}, 1)$-rooted grid model of order $x'$, say $M$. 
This means that there is a model function $\eta_1$ of $\cG_{x'}$ for $M$  such that for $i\in [\abs{W}]$, $V(\eta_1(v_{1,i}))$ contains a vertex of $W$.
We choose a vertex $w$ in $V(\eta_1(v_{x',x'}))$. 
We show that $w$ is an irrelevant vertex.
To show this, it is sufficient to check $\tau_{H}^*(G,\cZ, \ell)\le \tau_{H}^*(G- w, \cZ, \ell)$.
Let $T$ be a pure $(H, \cZ, \ell)$-deletion set of $G- w$.
We claim that $G$ contains a pure $(H, \cZ, \ell)$-deletion set of size at most $\abs{T}$.
If $G- T$ has no pure $(H, \cZ, \ell)$-models, then we are done.
We may assume that $G- T$ has a pure $(H, \cZ, \ell)$-model.
Let $T_A:=T\cap V(A)$ and $T_B:=T\cap V(B)$.

Let $W'$ be a minimum size subset of $W\setminus T$ such that $G- (T_A\cup W')$ contains no pure $(H, \cZ, \ell)$-models.
Such a set exists, as $G- (T_A\cup W)$ has no pure $(H, \cZ, \ell)$-models.
If $\abs{T\setminus V(A)}\ge \abs{W'}$, then $T_A\cup W'$ is a pure $(H, \cZ, \ell)$-deletion set of $G$ of size at most $\abs{T'}\le \abs{T}$.
So, we may assume that $\abs{T\setminus V(A)}\le \abs{W'}-1\le \abs{W}-1$.

\begin{figure}[t]
\centerline{\includegraphics[scale=0.7]{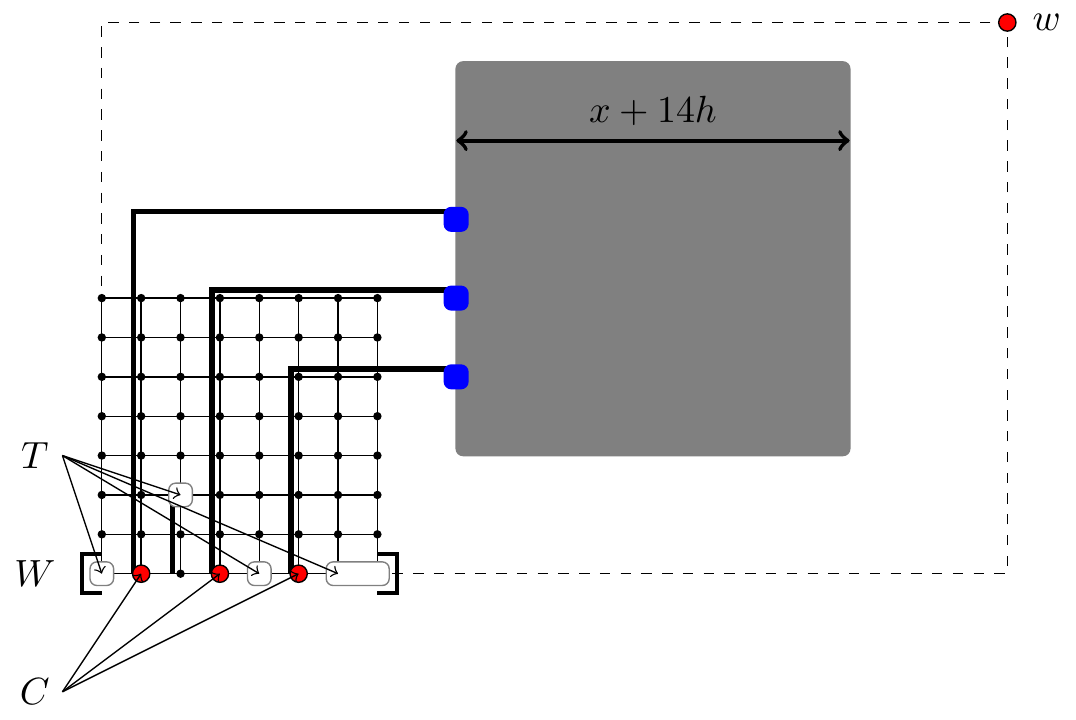}}
  \caption{The image of $\eta_1$ in Proposition~\ref{prop:irrelevant}, which is a grey rectangle. The three paths from $W$ to the subgrid are $Q_1, Q_2,$ and $Q_3$.}
  \label{fig:irrelevant}
\end{figure}

We prove that $G- (T\cup \{w\})$ has a pure $(H, \cZ, \ell)$-model, which yields a contradiction.
\begin{claim}
$G- (T\cup \{w\})$ has a pure $(H, \cZ, \ell)$-model.
\end{claim}
\begin{clproof}
Note that all vertices of $W$ are contained in the first row of $M$ and $w$ is contained in the last column of $M$.
See Figure~\ref{fig:irrelevant}.

Observe that for every $i\in [\abs{W}+1, x']$, the $i$-th column of $M$ does not contain a vertex of $W$.
Since $w$ is in the last column of $M$, there are at most $\abs{T\setminus V(A)}\le x-1$ columns of $M$ from the $(\abs{W}+1)$-th column to the $(x'-1)$-th column, 
that have a vertex in $T\cup \{w\}$. Because 
\[x'-\abs{W}-\abs{T\setminus V(A)}-1\ge x(x+14h+2)-2x=x(x+14h),\]
there is a set of $x+14h$ consecutive columns of $M$ that have no vertices in $T\cup W\cup \{w\}$.

A similar argument holds for rows as well. So, there is a set of $x+14h$ consecutive rows of $M$ that have no vertices in $T\cup W\cup \{w\}$.
It implies that 
there exist $1\le p\le x'-1-(x+14h)$ and $x\le q\le  x'-1-(x+14h)$ such that
for each $1\le i,j\le x+14h$, the $(p+i)$-th row and the $(q+j)$-th column do not contain a vertex of $T\cup W\cup \{w\}$.
We will use the grid model induced by 
$V(\eta_1(v_{p+i,q+j}))$'s for $1\le i,j\le x+14h$.
Let $G'$ be this subgraph. Note that $G'$ contains an $H$-model as $G'$ has order at least $14h$.

Now, we observe that there are $\abs{W}$ vertex-disjoint paths from $W$ to 
\[\eta (\{v_{p+1, q+1}\}), \ldots, \eta(\{v_{p+\abs{W}, q+1}\}).\]
We construct an $i$-th path such that it starts in $\eta_1(v_{1,i})\cap W$, goes to $\eta_1(v_{p+1+\abs{W}-i,i})$, and then terminates in $\eta_1(v_{p+1+\abs{W}-i, q+1}) $.
Among those paths, 
there are $\abs{W} - \abs{T_B}$ paths $Q_1, \ldots, Q_{\abs{W} - \abs{T_B}}$ from $W\setminus T$ to $G'$ avoiding $T\cup \{w\}$.
Let $C$ be the set of the end vertices of $Q_1, \ldots, Q_{\abs{W} - \abs{T_B}}$ in $W\setminus T$, and
let $C':=(W\setminus T)\setminus C$.

Since 
\[\abs{C}=\abs{W}-\abs{T_B}= \abs{W\setminus T}-\abs{T\setminus V(A)}\ge \abs{W\setminus T}-(\abs{W'}-1),\] we have $\abs{C'}\le \abs{W'}-1$. 
Therefore, by the choice of the set $W'$,
$G- (T_A\cup C')$ contains a pure $(H, \cZ, \ell)$-model $F$. 
If $F$ is contained in $A$, then it contradicts the assumption that $(G-w)-T$ has no pure $(H, \cZ, \ell)$-models.
Thus, $F$ contains a vertex of $B-V(A)$.
Note that each connected component of $F$ cannot be fully contained in $B- V(A)$ because $\multiabs{\cZ\setminus A}=0$. Thus
each connected component of $F$ containing a vertex of $B-V(A)$ also contains a vertex of $A$, 
and thus it contains a vertex of $W$.

Let $F'$ be the subgraph of $F$ consisting of all the connected components fully contained in $A- V(B)$.
Then there exist two disjoint sets $\cZ_1, \cZ_2\subseteq \cZ$ and a proper subset $\cC$ of the set of connected components of $H$  such that 
\begin{itemize}
\item $\abs{\cZ_1}+\abs{\cZ_2}=\ell$,
\item $F'$ is a $H[\bigcup_{C\in \cC} V(C)]$-model and 
$V(F')\cap Z\neq \emptyset$ for all $Z\in \cZ_1$, and
\item for each $Z\in \cZ_2$, there is a path from $Z$ to $W\setminus (T\cup C')$ in $A- (T\cup C')$ avoiding the vertices of $F'$.
\end{itemize}
Since $\cC$ is a proper subset of the set of connected components of $H$, 
we can obtain an $H'$-model in $G'$ where $H'$ consists of connected components of $H$ not contained in $\cC$.
Since all vertices of $W\setminus (T\cup C')$ are linked to $G'$ by the paths $Q_1, \ldots, Q_{\abs{W}- \abs{T_B}}$,
we have a pure $(H, \cZ, \ell)$-model, which is contained in $G- (T\cup \{w\})$. 
\end{clproof}
However, it contradicts our assumption that $G-(T\cup \{w\})$ has no pure $(H, \cZ, \ell)$-model.
Therefore, $\tau_{H}^*(G,\cZ, \ell)\le \tau_{H}^*(G- w, \cZ, \ell)$ and we conclude that $w$ is an irrelevant vertex.
\end{proof}

If $H$ is connected and $\ell=2$, then we further analyze the case when $G$ has a separation $(A,B)$ with $\multiabs{\cZ\setminus A}=1$. It will be needed for base cases. We remark that if $H$ is connected, then pure $(H, \cZ, 2)$-models are just $(H, \cZ, 2)$-models.

\begin{proposition}\label{prop:irrelevant2}
Let $g$ and $h$ be positive integers and $x$ be a non-negative integer such that $g\ge 2(4x^2+14hx+3x+1)(4x^2+1)$.
Let $(G,\cZ)$ be a rooted graph, $H$ be a connected planar graph with $h$ vertices,
and $(A,B)$ be a separation in $G$ of order at most $x$
such that $\multiabs{\cZ\setminus A}\le 1$ and $B-V(A)$ contains a $\cG_g$-model. 
Then $G$ contains an irrelevant vertex $w$ for pure $(H, \cZ, 2)$-models.
\end{proposition}
\begin{proof}
We prove it by induction on $x$.
If $x=0$, then no pure $(H, \cZ, 2)$-models have a vertex of $B- V(A)$ as $H$ is connected. 
Thus, every vertex in $B- V(A)$ is irrelevant, and
since $g\ge 1$,  there is an irrelevant vertex in $B- V(A)$.
Let us assume that $x\ge 1$.
Also, if $\multiabs{\cZ\setminus A}=0$, then by Proposition~\ref{prop:irrelevant}, 
$G$ contains an irrelevant vertex.
So, we may assume that $\multiabs{\cZ\setminus A}=1$.

Let $W:=V(A\cap B)$ and $x':=2(4x^2+14hx+3x)$.
Let $Z_a$ be the unique set in $\cZ$ that intersects $B-V(A)$,  and let $Y_a:=Z_a\setminus V(A)$ and $\cZ':=\{W, Y_a\}$.
Note that $g\ge x'(4x^2+1)+4x^2$.  
By applying Theorem~\ref{thm:rootedgridminor} to $B$ (with $(\cZ, g, k, \ell)\leftarrow (\cZ', x', \abs{W}, 2))$, 
$B$ contains either
\begin{enumerate}[(1)]
\item a separation $(A',B')$ of order less than $\abs{W}(2-\multiabs{\cZ'\setminus A'})$ in $B$ such that
$\multiabs{\cZ'\setminus A'}\le 1$ and $B'- V(A')$ contains $\cG_{g-\abs{V(A'\cap B')}}$-model, or
\item a $(\cZ', \abs{W}, 2)$-rooted grid model of order $x'$.
\end{enumerate}

Suppose there is a separation $(A',B')$ described in (1). Since $\multiabs{\cZ'\setminus A'}\le 1$, there are three possibilities; either 
($Y_a\setminus V(A')\neq \emptyset$ and $W\subseteq V(A')$), or
($Y_a\subseteq V(A')$ and $W\setminus V(A')\neq \emptyset$), or 
($Y_a\cup W\subseteq V(A')$).
In each case, we argue that there is an irrelevant vertex.

\begin{itemize}
\item (Case 1. $Y_a\setminus V(A')\neq \emptyset$ and $W\subseteq V(A')$.) \\
Then $((A- E(A\cap B))\cup A', B')$ is a separation in $G$ of order at most $\abs{W}-1\le x-1$, and $\multiabs{\cZ\setminus (A\cup A')}=1$ and $B'- V(A\cup A')$ contains a $\cG_{g-(x-1)}$-model.
Since 
\begin{align*}
g-(x-1)&\ge 2(4x^2+14hx+3x+1)(4x^2+1)-(x-1) \\
&\ge 2(4(x-1)^2+14h(x-1)+3(x-1)+1)(4(x-1)^2+1),
\end{align*}
by the induction hypothesis, 
$G$ contains an irrelevant vertex.
\item (Case 2. $Y_a\cup W\subseteq V(A')$.) \\
Then 
$((A- E(A\cap B))\cup A', B')$ is a separation in $G$ of order at most $2\abs{W}\le 2x-1$, and $\multiabs{\cZ\setminus (A\cup A')}=0$ and $B'- V(A\cup A')$ contains a $\cG_{g-(2x-1)}$-model.
Since 
\begin{align*}
g-(2x-1)&\ge 2(4x^2+14hx+3x+1)(4x^2+1)-(2x-1)\\
		&\ge (8x^2+28hx+6x+2)(4x^2+1)-(2x-1) \\
 		&\ge (4x^2+28hx)(4x^2+1) \\
 		&\ge (4x^2+28hx-14h)(4x^2-4x+2) \\
 		&= ((2x-1)^2+14h(2x-1)+2(2x-1)+1)((2x-1)^2+1),
\end{align*}
by Proposition~\ref{prop:irrelevant}, $G$ contains an irrelevant vertex.
\item (Case 3. $Y_a\subseteq V(A')$ and $W\setminus V(A')\neq \emptyset$.) \\
In this case, 
$((A- E(A\cap B))\cup A', B')$ is a separation in $G$ of order at most
$\abs{W\cup V(A'\cap B')}\le 2x-1$ where $\multiabs{\cZ\setminus (A\cup A')}=0$ and $B'- V(A\cup A')$ contains a $\cG_{g-(2x-1)}$-model.
By the same reason in Case 2, $G$ contains an irrelevant vertex.
\end{itemize}

\vskip 0.3cm
Now, we assume that $B$ contains  a $(\cZ', \abs{W}, 2)$-rooted grid model of order $x'$, say $M$.
So, there is a model function $\eta_1$ of $\cG_{x'}$ for $M$ such that for $1\le i\le 2\abs{W}$, $V(\eta_1(v_{1,i}))$ contains a vertex of $W\cup Y_a$ and $W\subseteq \{w_1, \ldots, w_{2\abs{W}}\}$.
We choose a vertex $w$ in $V(\eta_1(v_{x',x'}))$. 
We show that $w$ is an irrelevant vertex.
To show this, it is sufficient to check that $\tau_{H}^*(G,\cZ, 2)\le \tau_{H}^*(G- w, \cZ, 2)$.
Let $T$ be a pure $(H, \cZ, 2)$-deletion set $T$ of $G- w$. 
We claim that $G$ contains a pure $(H, \cZ, \ell)$-deletion set of size at most $\abs{T}$.
If $G- T$ has no pure $(H, \cZ, 2)$-models, then we are done.
We may assume that $G- T$ has a pure $(H, \cZ, 2)$-model.
Let $T_A:=T\cap V(A)$ and $T_B:=T\cap V(B)$.

Let $W'$ be a minimum size subset of $W\setminus T$ such that $G- (T_A\cup W')$ contains no pure $(H, \cZ, 2)$-models.
Such a set exists, because $\multiabs{\cZ\setminus A}=1$ and thus $G- (T_A\cup W)$ has no $(H, \cZ, 2)$-models.
If $\abs{T\setminus V(A)}\ge \abs{W'}$, then $T_A\cup W'$ is a pure $(H, \cZ, 2)$-deletion set of size at most $\abs{T}$.
So, we may assume that $\abs{T\setminus V(A)}\le \abs{W'}-1$.

We claim that $G- (T\cup \{w\})$ has a pure $(H, \cZ, 2)$-model, which yields a contradiction.

\begin{claim}
$G- (T\cup \{w\})$ has a pure $(H, \cZ, 2)$-model.
\end{claim}
\begin{clproof}
Note that all vertices of $W$ are contained in the first row of $M$ and $w$ is contained in the last column of $M$.

Since $w$ is in the last column of $M$, there are at most $\abs{T\setminus V(A)}\le x-1$ columns of $M$ from the $(2\abs{W}+1)$-th column to the $(x'-1)$-th column, 
contains vertices of $T\setminus V(A)$. Because 
\begin{align*}
x'-2\abs{W}-\abs{T\setminus V(A)}-1&\ge 2x(4x+14h+3)-3x \\
&\ge x(x+14h),
\end{align*}
there is a set of $x+14h$ consecutive columns of $M$ that have no vertices in $T\cup \{w_1, w_2, \ldots, w_{2\abs{W}}, w\}$.

A similar argument holds for rows as well. So, there is a set of $x+14h$ consecutive rows of $M$ that have no vertices in $T\cup W\cup \{w\}$.
It implies that 
there exist $1\le p\le x'-1-(x+14h)$ and $x\le q\le  x'-1-(x+14h)$ such that
for each $1\le i,j\le x+14h$, the $(p+i)$-th row and the $(q+j)$-th column do not contain a vertex of $T\cup W\cup \{w\}$.
We will use the grid model induced by 
$V(\eta_1(v_{p+i,q+j}))$'s for $1\le i,j\le x+14h$.
Let $G'$ be this subgraph.

Observe that there are $2\abs{W}$ vertex-disjoint paths from $\{w_1, \ldots, w_{2\abs{W}}\}$ to 
$\eta (\{v_{p+1, q+1}\}), \ldots, \eta(\{v_{p+2\abs{W}, q+1}\})$.
Among those paths starting from $W$, 
there are at least $\abs{W}-\abs{T_B}$ paths $Q_1, \ldots, Q_{\abs{W}-\abs{T_B}}$ from $W\setminus T$ to $G'$ avoiding $T\cup \{w\}$.
Similarly, among those paths starting from $\{w_1, \ldots, w_{2\abs{W}}\}\setminus W\subseteq Y_a$,  there are at least $\abs{W}-\abs{T\setminus V(A)}\ge 1$ paths from $\{w_1, \ldots, w_{2\abs{W}}\}\setminus W$ to $G'$, avoiding $T\cup \{w\}$.
Let $R$ be the one of the latter paths. Note that $R$ connects the set $Z_a$ and $G'$.

Let $C$ be the set of the end vertices of paths $Q_1, \ldots, Q_{\abs{W}-\abs{T_B}}$ contained in $W\setminus T$, and
let $C':=(W\setminus T)\setminus C$.
Since 
\[\abs{C}=\abs{W}-\abs{T_B}= \abs{W\setminus T}-\abs{T\setminus V(A)}\ge \abs{W\setminus T}-(\abs{W'}-1),\] we have $\abs{C'}\le \abs{W'}-1$. 
Therefore, by the choice of the set $W'$,  
there is an $(H, \cZ, 2)$-model in $G- (T_A\cup C')$.
In particular, this model should intersect a set in $\cZ\setminus \{Z_a\}$.
As $H$ is connected, there is a path $P$ from $\bigcup_{Z\in \cZ\setminus \{Z_a\}} Z$ to $G'$ avoiding $T_A$.
Therefore,
$G'\cup P\cup R$ contains an $(H, \cZ, 2)$-model, which is contained in $G- (T\cup \{w\})$. 
\end{clproof}

However, it contradicts our assumption that $G- (T\cup \{w\})$ contains no pure $(H, \cZ, 2)$-model.
Therefore, $\tau_{H}^*(G,\cZ, 2)\le \tau_{H}^*(G- w, \cZ, 2)$ and we conclude that $w$ is an irrelevant vertex.
\end{proof}

\subsection{Separating a grid model from sets of $\cZ$}
 
In the proof of Theorem~\ref{thm:mainpure}, we will proceed by induction on $\ell$.
The following proposition shows that given a sufficiently large grid-model, we can find either many disjoint $(H, \cZ, \ell)$-models, or a separation that separates a large grid-model from most of sets in $\cZ$. The main difference from Theorem~\ref{thm:rootedgridminor} is that here we also guarantee that when we have a latter separation, we keep having a rooted grid model with a smaller subset of $\cZ$. This will help to reduce the instance into an instance with smaller $\ell$ value so that one can apply the induction hypothesis. 

\begin{proposition}\label{prop:excluding}
Let $g,k,h$, and $\ell$ be positive integers with $g\ge 2(k\ell(14h+2)+1)(k^2\ell^2+1)$ and let 
$\ell^*:=\ell-1$ if $\ell\ge 2$, and $\ell^*:=\ell$ otherwise.
Every rooted graph $(G, \cZ)$ having a $\cG_g$-model contains either
\begin{enumerate}
\item $k$ pairwise vertex-disjoint $(\ell^*\cdot \cG_{14h}, \cZ, \ell)$-models, or
\item a separation $(A,B)$ of order less than $k\ell^2$  such that 
 $\multiabs{\cZ\setminus A}<\ell$ and 
$B- V(A)$ contains a $\cG_{g-k\ell^2}$-model and also contains a $(\cZ\setminus A,k, \multiabs{\cZ\setminus A})$-rooted grid model of order $k\ell(14h+2)+1$.
\end{enumerate}
\end{proposition}
\begin{proof}
We start with finding a separation $(A_0,B_0)$ in $G$ of order less than $k\ell$ where $\multiabs{\cZ\setminus A_0}\le \ell-1$ and $B_0- V(A_0)$ contains a $\cG_{g-k\ell}$-model.
If $\multiabs{\cZ}\le \ell-1$, then the separation $(\emptyset, G)$ is such a separation.
Suppose that $\multiabs{\cZ}\ge \ell$.
Since $g\ge (k\ell(14h+2)+1)(k^2\ell^2+1)+k\ell$,
by Theorem~\ref{thm:rootedgridminor},
$G$ contains either 
 such a separation $(A_0, B_0)$, or
a $(\cZ, k, \ell)$-rooted grid model of order $k\ell(14h+2)+1$.
If it has a rooted grid model,
then by Lemma~\ref{lem:colorfulgrid}, 
$G$ contains $k$ pairwise vertex-disjoint $(\ell^*\cdot \cG_{14h}, \cZ, \ell)$-models.
Therefore, we may assume that 
there is a separation $(A_0, B_0)$ in $G$ of order less than $k\ell$ where $\multiabs{\cZ\setminus A_0}\le \ell-1$ and $B_0- V(A_0)$ contains 
a $\cG_{g-k\ell}$-model.

We show the following.

\begin{claim}	
	$G$ contains a separation $(A,B)$ with order less than  $k\ell^2$ such that 
	$\multiabs{\cZ\setminus A}=\ell'$ for some $0\le \ell'<\ell$, 
    and 
	$B- V(A)$  contains a $\cG_{g-k\ell^2}$-model and also contains a $(\cZ\setminus A,k,\ell')$-rooted grid model of order $k\ell(14h+2)+1$.
\end{claim}
\begin{clproof}
We recursively construct a sequence 
$(A_0, B_0), \ldots, (A_{\ell-1}, B_{\ell-1}) $ such that  for each $0\le i\le \ell-1$, 
	\begin{itemize}
	\item $(A_i, B_i)$ is a separation in $G$ of order less than $k\sum_{0\le j\le i}(\ell-j)$, 
	\item $\multiabs{\cZ\setminus A_{i}}\le (\ell-1)-i$, and
\item  $B_{i}- V(A_{i})$ contains a $\cG_{g-k\sum_{0\le j\le i}(\ell-j)}$-model, 
\end{itemize}
unless the separation described in the claim exists.
If there is such a sequence, then $(A_{\ell-1}, B_{\ell-1})$ is a separation of order less than \[k\sum_{0\le j\le \ell-1}(\ell-j)=k\sum_{1\le j\le \ell}j=\frac{k\ell(\ell+1)}{2}\le k\ell^2\] such that
$\multiabs{\cZ\setminus A_{\ell-1}}=0$ and $B_{\ell-1}- V(A_{\ell-1})$ contains a $\cG_{g-k\ell^2}$-model.
Since a $\cG_{g-k\ell^2}$-model is also a $(\emptyset, k, 0)$-rooted grid model of order at least $k\ell(14h+2)+1$,
$(A_{\ell-1}, B_{\ell-1})$ is a required separation.

Note that the sequence $(A_0, B_0)$ exists.
Suppose that there is a such a sequence $(A_0, B_0), \ldots, (A_{t}, B_{t})$ exists for some $0\le t<\ell-1$.
If $\multiabs{\cZ\setminus A_{t}}< (\ell-1)-t$, then by taking $(A_{t+1}, B_{t+1}):=(A_{t}, B_{t})$, we have a required sequence $(A_0, B_0), \ldots, (A_{t+1}, B_{t+1})$.
Thus, we may assume that 
$\multiabs{\cZ\setminus A_{t}}= (\ell-1)-t$.
Let $\cY:=\cZ\setminus A_{t}$.
Note that $\multiabs{\cY}=(\ell-1)-t$.

By Theorem~\ref{thm:rootedgridminor},
$B_t-V(A_t)$ contains
either
\begin{enumerate}[(1)]
\item a separation $(C,D)$ of order less than $k(\multiabs{\cY}-\multiabs{\cY\setminus C})$ where $\multiabs{\cY\setminus C}\le \multiabs{\cY}-1$ and $D- V(C)$ contains a $\cG_{g-k\sum_{0\le j\le t}(\ell-j)-\abs{V(C\cap D)}}$-model, or
\item a $(\cY, k, \multiabs{\cY})$-rooted grid model of order $k\ell(14h+2)+1$.
\end{enumerate}
If it has the latter rooted grid model,
then $(A_t, B_t)$ is a required separation where $\ell'=(\ell-1)-t$ and $\cZ\setminus A_t=\cY$, 
because $B_t-V(A_t)$ also contains a $\cG_{g-k\ell^2}$-model.

Assume that
we have a separation $(C,D)$ described in (1).
Let $A_{t+1}:=G[V(A_t)\cup V(C)]$ and $B_{t+1}:=G[V(D)\cup V(A_t\cap B_t)]- E(G[V(A_t\cap B_t)\cup V(C\cap D)])$.
Observe that $(A_{t+1}, B_{t+1})$ is a separation in $G$ such that 
\begin{itemize}
\item $\multiabs{\cZ\setminus A_{t+1}}\le \multiabs{\cY}-1=\ell-1-(t+1)$, and
\item $B_{t+1}- V(A_{t+1})$ contains a $\cG_{g-k\sum_{0\le j\le t}(\ell-j)-\abs{V(C\cap D)}}$-model. 
\end{itemize}
Since $\abs{V(C\cap D)}<k\multiabs{\cY}=k(\ell-1-t)$, $\abs{V(A_t\cap B_t)}+k(\ell-1-t)<k\sum_{0\le j\le t+1}(\ell-j)$, as required.
\end{clproof}

We conclude that $G$ contains either $k$ pairwise vertex-disjoint $(\ell^*\cdot \cG_{14h}, \cZ, \ell)$-models, or
a separation $(A,B)$ with order less than $k\ell^2$  such that 
$\multiabs{\cZ\setminus A}=\ell'$ for some $0\le \ell'<\ell$ and 
$B- V(A)$ contains a $(\cZ\setminus A,k, \ell')$-rooted grid model of order $k\ell(14h+2)+1$.
\end{proof}

\begin{lemma}\label{lem:reduction}
Let  $\ell$ be a positive integer.
Let $(G, \cZ)$ be a rooted graph such that
there is a separation $(A,B)$ in $G$  such that 
$\multiabs{\cZ\setminus A}=\ell'$ for some $1\le \ell'<\ell$. Let $\cZ'$ be the multiset of all sets $X$ in $\cZ$ where $X\subseteq V(A)$.
If $T$ is a pure $(H, \cZ', \ell-\ell')$-deletion set of $A- V(B)$,
then  $T\cup V(A\cap B)$ is a pure $(H, \cZ, \ell)$-deletion set of $G$.
\end{lemma}
\begin{proof}
Suppose that $G- (T\cup V(A\cap B))$ has a pure $(H, \cZ, \ell)$-model $F$.
Since $\abs{\cZ\setminus A}=\ell'$,
$F\cap (A- V(B))$ should meet at least $\ell-\ell'$ sets of $\cZ'$.
It means that $F\cap (A- V(B))$ is a pure $(H, \cZ', \ell-\ell')$-model.
Since $T$ meets all such models, it contradicts our assumption.
We conclude that  $G- (T\cup V(A\cap B))$ has no $(H, \cZ, \ell)$-models.
\end{proof}

Now, we are ready to give the proof of Theorem~\ref{thm:mainpure}. 

\begin{proof}[Proof of Theorem~\ref{thm:mainpure}]
We recall that $\kappa$ is the function in the Grid Minor Theorem. 
For a positive integer $\ell$ and a non-empty planar graph $H$ with $\cc(H)\ge \ell-1$ and $h=V(H)$,
we define that 
\begin{align*}
x=x_{H, \ell}(k)&:=k\ell^2 \\
g=g_{H, \ell}(k)&:=2(4x^2+14hx+3x+1)(4x^2+1)+x\\
f^1_{H, \ell}(k)&:=\left\{ \begin{array}{ll}
 \kappa(g_{H, \ell}(k))(h^2-h+1)(k-1) & \textrm{if $\ell=1$}\\
 \kappa(g_{H, \ell}(k))(h^2-h+1)(k-1) \\ \qquad + (\ell-1) f^1_{H, \ell-1}(k)+ k\ell^2 & \textrm{if $\ell\ge 2$.}
 \end{array} \right.
\end{align*}
Note that $g_{H, \ell}(k)\ge 2(14hx+2x+1)(x^2+1)$, the function defined in Proposition~\ref{prop:excluding}.
We will use the fact that $\cG_{14h}$ contains an $H$-model.

We prove the statement by induction on $\ell+\cc(H)$.
Suppose that the theorem does not hold, and 
$G$ is a minimal counterexample.

We first show for two base cases.

\begin{itemize}
\item (Case 1-1. $\ell=1$) \\
If $G$ has tree-width at most $\kappa(g)$, then 
by Proposition~\ref{prop:eponboundedtw},
$G$ contains $k$ pairwise vertex-disjoint pure $(H, \cZ, 1)$-models or a pure $(H, \cZ, 1)$-deletion set of size at most $(\kappa(g)+1)(h^2-h+1)(k-1)$. 
Thus, we conclude that 
$\tau^*_H(G, \cZ, 1)\le (\kappa(g)+1)(h^2-h+1)(k-1)=f^1_{H, \ell}(k).$

Suppose $G$ has tree-width larger than $\kappa(g)$.
By Theorem~\ref{thm:gridtheorem}, $G$ contains a $\cG_{g}$-model.
So, by Proposition~\ref{prop:excluding}, 
$G$ contains either 
$k$ pairwise vertex-disjoint $(H, \cZ, 1)$-models, or
a separation $(A,B)$ with order less than $x$  such that 
$\multiabs{\cZ\setminus A}=0$ and 
$B- V(A)$ contains a $\cG_{g-x}$-model.
As each $(H, \cZ, 1)$-model contains a pure $(H, \cZ, 1)$-model, 
in the former case, we have $k$ pairwise vertex-disjoint pure $(H, \cZ, 1)$-models.
Thus, we may assume that we have the latter separation $(A,B)$.
Then, since
\[ g-x\ge (x^2+14hx+2x+1)(x^2+1),\]
by Proposition~\ref{prop:irrelevant}, 
there exists an irrelevant vertex $v$ for pure $(H, \cZ, 1)$-models in $G$. 
This contradicts the fact that $G$ is chosen to be a minimal counterexample.
\item (Case 1-2. $\ell=2$ and $\cc(H)=1$.)
This case is almost same as Case 1, but we use 
Proposition~\ref{prop:irrelevant2} instead of Proposition~\ref{prop:irrelevant} to find an irrelevant vertex.
This is possible because $g-x\ge 2(4x^2+14hx+3x+1)(4x^2+1)$.
\end{itemize}

Now we assume that $\ell\ge 2$ and $\cc(H)\ge 2$ if $\ell=2$.
We may assume that $G$ has tree-width at least $\kappa(g)$, otherwise we can apply Proposition~\ref{prop:eponboundedtw}. 
Now, we reduce the given instance in two ways; either reduce its tree-width or the parameter $\ell$.
By Theorem~\ref{thm:gridtheorem}, $G$ contains a $\cG_{g}$-model.

By Proposition~\ref{prop:excluding}, 
$G$ contains either $k$ pairwise vertex-disjoint $(H, \cZ, \ell)$-models, or
a separation $(A,B)$ with order less than $x$  such that 
\begin{itemize}
\item $\multiabs{\cZ\setminus A}=\ell'<\ell$ and 
$B- V(A)$ contains a $\cG_{g-x}$-model and a $(\cZ\setminus A,k,\ell')$-rooted grid model of order $k\ell(14h+2)+1$.
\end{itemize}
We may assume that we have the latter separation.
Let $\cZ'$ be the multiset of all sets $X$ in $\cZ$ where $X\subseteq V(A)$.

We observe that 
if $\ell'=0$, then there is an irrelevant vertex $v$ by Proposition~\ref{prop:irrelevant}.
Thus, in this case, $G-v$ satisfies the theorem because $G$ is chosen as a minimal counterexample. 
But then $G$ also satisfies the theorem as $\tau^*_H(G, \cZ, \ell)=\tau^*_H(G-v, \cZ, \ell)$.
Thus, we can assume that $\ell'\ge 1$.
We will argue that in the remaining part, one can reduce the instance into $A- V(B)$ with the parameter $\ell-\ell'$.

Recall that $\cc(H)\ge \ell-1$. We divide into two cases depending on $\cc(H)\ge \ell$ or not. 

\begin{itemize}
\item (Case 2-1. $\cc(H)\ge \ell$.) \\
Since $B- V(A)$ contains a $(\cZ\setminus A,k, \ell')$-rooted grid model of order $k\ell(14h+2)+1$, 
by Lemma~\ref{lem:colorfulgrid},
$B- V(A)$ contains $k$ pairwise vertex-disjoint $(\ell'\cdot \cG_{14h}, \cZ\setminus A, \ell')$-models. 
We will use this later.

Since $\ell'>0$, we have $\ell-\ell'<\ell$.
So, we apply the induction hypothesis on ($A- V(B)$, $\cZ'$, $\ell-\ell'$, $k$), 
and we have that $A-V(B)$ contains either $k$ vertex-disjoint pure $(H, \cZ', \ell-\ell')$-models, or
a pure $(H, \cZ', \ell-\ell')$-deletion set $T$ of size at most $f^1_{H, \ell-\ell'}(k)$.
If it outputs a deletion set $T$, 
then by Lemma~\ref{lem:reduction}, $G- \left(T\cup V(A\cap B)\right)$
 has no pure $(H, \cZ,\ell)$-models.
Since 
\[f^1_{H, \ell-\ell'}(k) + k\ell^2\le f^1_{H, \ell-1}(k) +  k\ell^2 \le f^1_{H, \ell}(k), \]
$T\cup V(A\cap B)$ is a required pure $(H, \cZ, \ell)$-deletion set in $G$.

Suppose we have $k$ pairwise vertex-disjoint pure $(H, \cZ', \ell-\ell')$-models in $A- V(B)$.
Since each model in $A- V(B)$ consists of the image of at most $\ell-\ell'$ connected components of $H$ and $H$ consists of at least $\ell$ connected components, 
we can complete it into a pure $(H, \cZ, \ell)$-model 
by taking a relevant pure $(H, \cZ\setminus A, \ell')$-model from a $(\ell'\cdot \cG_{14h}, \cZ\setminus A, \ell')$-model in $B- V(A)$.
For instance, if $H_1, \ldots, H_{\ell}$ is a set of connected components of $H$ and a pure $(H, \cZ', \ell-\ell')$-models in $A- V(B)$  is the image of the union of $H_1, \ldots, H_{\ell-\ell'}$, then we obtain images of $H_{\ell-\ell'+1}, \ldots, H_{\ell}$ from each connected component of the $\ell'\cdot \cG_{14h}$-grid model.
Therefore, $G$ contains $k$ pairwise vertex-disjoint pure $(H, \cZ, \ell)$-models.
\item (Case 2-2. $\cc(H)= \ell-1$. ) \\
When $\ell'\ge 2$, 
we can prove as in Case 2-1.
Note that every pure $(H, \cZ', \ell-\ell')$-model is the image of at most $\ell-\ell'$ connected components of $H$, and
there are $k$ pairwise vertex-disjoint pure $((\ell'-1)\cdot \cG_{14h}, \cZ\setminus A, \ell')$-models in $B- V(A)$ by Lemma~\ref{lem:colorfulgrid}.
Thus, we can return $k$ pairwise vertex-disjoint pure $(H, \cZ, \ell)$-models, or a pure $(H, \cZ, \ell)$-deletion set of size at most $f^1_{H, \ell}(k)$.
But, we cannot do the same thing when $\ell'=1$, because we can only say that $B- V(A)$ contains $k$ pairwise vertex-disjoint pure $(\ell'\cdot \cG_{14h}, \cZ\setminus A, \ell')$-models, not  pure $((\ell'-1)\cdot \cG_{14h}, \cZ\setminus A, \ell')$-models.
So, we may assume that $\ell'=1$.

If $\ell=2$, then $H$ is connected, and this case was resolved in Case 1-2.
So, we may also assume that $\ell\ge 3$.

Note that a pure $(H, \cZ', \ell-1)$-model in $A- V(B)$ may be the image of $H$ itself.
But, this cannot be a part of a pure $(H, \cZ, \ell)$-model of $G$, and thus, we ignore it (the model already used all components to meet only $\ell-1$ sets of $\cZ$).
For this reason, we only consider pure $(H, \cZ', \ell-1)$-models in $A- V(B)$ that are not the images of $H$.

For each subgraph $H'$ of $H$ induced by its $\ell-2$ connected components, we apply induction hypothesis with $(A- V(B), \cZ', \ell-1, k)$.
Since $\cc(H')<\cc(H)$, we obtain that $A-V(B)$ contains either $k$ pairwise vertex-disjoint pure $(H', \cZ', \ell-1)$-models
or a pure $(H', \cZ', \ell-1)$-deletion set $T_{H'}$ of size at most $f^1_{H, \ell-1}(k)$.

Suppose that for some subgraph $H'$,  we have $k$ pairwise vertex-disjoint pure $(H', \cZ', \ell-1)$-models in $A- V(B)$.
Then we can complete them into $k$ pairwise vertex-disjoint pure $(H, \cZ, \ell)$-models in $G$ using $k$ pairwise vertex-disjoint $(\cG_{14h}, \cZ\setminus A, 1)$-models in $B- V(A)$.
Therefore, we may assume that for all possible subgraphs $H'$ of $H$ with $\ell-2$ connected components, 
we have a pure $(H', \cZ', \ell-1)$-deletion set $T_{H'}$  in $A- V(B)$.
Let $T$ be the union of all such deletion sets $T_{H'}$. Note that $\abs{T}\le (\ell-1)f^1_{H, \ell-1}(k)$. 

We claim that 
$G- \left(T\cup V(A\cap B)\right)$
 has no pure $(H, \cZ,\ell)$-models.
Suppose $G- \left(T\cup V(A\cap B)\right)$ has a pure $(H, \cZ,\ell)$-model $F$.
First assume that $F$ is fully contained in $A- V(B)$. 
If $F$ is the image of a subgraph $H'$ of $H$ consisting of its at most $\ell-2$ connected components, then 
there exists a subgraph $H''$ of $H$ induced by its exactly $\ell-2$ connected components where $H'$ is a subgraph of $H''$. 
Thus, by ignoring the set in $\cZ$ intersecting $B-V(A)$, $F$ contains a $(H'', \cZ', \ell-1)$-model.
But $T_{H''}$ meets a vertex of $F$, contradicting the assumption that $V(F)\cap T=\emptyset$.
Thus, we may assume that $F$ is the image of $H$. 
Since $\ell\ge 3$ and $H$ consists of $\ell-1$ connected components and $F$ intersects $\ell$ sets of $\cZ$, 
there should be a connected component of $F$ intersecting exactly one set among those $\ell$ sets.
In other words, there is a subgraph $H'$ of $H$ consisting of its $\ell-2$ connected components where 
the subgraph of $F$ induced by the image of $H'$ intersects $\ell-2$ sets of $\cZ$.
But this contradicts our assumption that $V(F)\cap T_{H'}=\emptyset$.

So, we may assume that $F-V(A)\neq \emptyset$. In this case, since $\ell'=1$, only one connected component of $F$ can be contained in $B- V(A)$.
So, there are at most $\ell-2$ connected components of $F\cap (A- V(B))$ whose union meets $\ell-1$ sets of $\cZ'$.
It means that $F\cap (A- V(B))$ is a pure $(H', \cZ', \ell-1)$-model for some 
subgraph $H'$ of $H$ induced by its $\ell-2$ connected components.
Since $T$ meets all such models, we have a contradiction.

We conclude that $G- \left(T\cup V(A\cap B)\right)$ has no pure $(H, \cZ,\ell)$-models. 
Since 
$(\ell-1) f^1_{H, \ell-1}(k) + k\ell^2\le f^1_{H, \ell}(k)$,
we have a required pure $(H, \cZ, \ell)$-deletion set.
\end{itemize}
We conclude that $G$ contains either
$k$ pairwise vertex-disjoint pure $(H, \cZ, \ell)$-models, or
a pure $(H, \cZ, \ell)$-deletion set of size at most $f^1_{H, \ell}(k)$.
\end{proof}

\section{Packing and covering $(H, \cZ, \ell)$-models}\label{sec:total}

Now we prove the main result of this paper.

\begin{thmmain}
For a positive integer $\ell$ and a non-empty planar graph $H$ with $\cc(H)\ge \ell-1$,
there exists $f_{H, \ell}:\mathbb{N}\rightarrow \mathbb{R}$ satisfying the following property.
Let $(G, \cZ)$ be a rooted graph and $k$ be a positive integer.
Then $G$ contains either
$k$ pairwise vertex-disjoint $(H, \cZ, \ell)$-models in $G$, or
an $(H, \cZ, \ell)$-deletion set of size at most  $f_{H,\ell}(k)$.
\end{thmmain}

\begin{proof}
We recall from the proof of Theorem~\ref{thm:mainpure} functions $x, g$, and $f^1$. For all positive integers $k, \ell$, and $h$, 
\begin{align*}
x=x_{H, \ell}(k)&:=k\ell^2 \\
g=g_{H, \ell}(k)&:=2(4x^2+14hx+3x+1)(4x^2+1)+x\\
f^1_{H, \ell}(k)&:=\left\{ \begin{array}{ll}
 \kappa(g_{H, \ell}(k))(h^2-h+1)(k-1) & \textrm{if $\ell=1$}\\
 \kappa(g_{H, \ell}(k))(h^2-h+1)(k-1) \\ \qquad + (\ell-1) f^1_{H, \ell-1}(k)+ k\ell^2 & \textrm{if $\ell\ge 2$.}
 \end{array} \right. \\
f_{H, \ell}(k)&:=\ell f^1_{H, \ell}(k)+ k\ell^2.
\end{align*}

We observe that if $\cc(H)=1$, then pure $(H, \cZ, \ell)$-models are exactly $(H, \cZ, \ell)$-models.
Therefore, Theorem~\ref{thm:mainpure} implies the statement.
We may assume that $\cc(H)\ge 2$.
 
If $G$ has tree-width at most $\kappa (g)$, and 
by Proposition~\ref{prop:eponboundedtw},
$G$ contains either $k$ pairwise vertex-disjoint $(H, \cZ, \ell)$-models or an $(H, \cZ, \ell)$-deletion set $T$ of size at most 
$(\kappa (g)+1)(h^2-h+1)(k-1)\le f^1_{H, \ell}(k)\le f_{H, \ell}(k)$.
Therefore, we may assume that $G$ has tree-width larger than $\kappa(g)$.
By Theorem~\ref{thm:gridtheorem}, $G$ contains a $\cG_g$-model.

By Proposition~\ref{prop:excluding},
$G$ contains either $k$ pairwise vertex-disjoint $(H, \cZ, \ell)$-models,
or a separation $(A,B)$ with order less than $x$  such that 
\begin{itemize}
\item $\multiabs{\cZ\setminus A}=\ell'$ for some $0\le \ell'<\ell$ and 
$B- V(A)$ contains a $\cG_{g-x}$-model, and contains a $(\cZ\setminus A,k, \ell')$-rooted grid model of order $k\ell(14h+2)+1$.
\end{itemize}
We may assume that we have the latter separation.

If $\ell'\ge 1$, then we do a procedure similar to the proof in Theorem~\ref{thm:mainpure}.
But, when $\ell'=0$ there was an irrelevant vertex argument for pure models. 
That argument cannot be extended to usual models.
Instead, we can reduce the instance into an instance of  packing and covering $k$ pairwise disjoint pure $(H, \cZ, \ell)$-models in $A- V(B)$.

\begin{itemize}
\item (Case 1. $\ell'=0$.) \\
We apply Theorem~\ref{thm:mainpure} to the instance $(A- V(B), \cZ, \ell, k)$ for pure $(H, \cZ, \ell)$-models.
Then $A-V(B)$ contains either 
$k$ pairwise vertex-disjoint pure $(H, \cZ, \ell)$-models, 
or
a pure $(H, \cZ, \ell)$-deletion set $T$ of size at most $f^1_{H, \ell}(k)$.

Suppose that $A-V(B)$ has $k$ pairwise vertex-disjoint pure $(H, \cZ, \ell)$-models.
Note that among those models, some of them is the image of $H$ itself, and some is not the image of $H$.
Since $g-x\ge 14hk$, $B-V(A)$ contains $k$ vertex-disjoint $\cG_{14h}$-models.
So, for those pure models in $A-V(B)$ that are not the images of $H$, we can complete them into $(H, \cZ, \ell)$-models, using $\cG_{14h}$-models in $B-V(A)$.
Therefore, $G$ has $k$ pairwise vertex-disjoint $(H, \cZ, \ell)$-models.

On the other hand, if we have a pure $(H, \cZ, \ell)$-deletion set $T$, then 
$G- (T\cup V(A\cap B))$ has no $(H, \cZ, \ell)$-models.
This is because every  $(H, \cZ, \ell)$-model in $G- V(A\cap B)$ 
contains a pure $(H, \cZ, \ell)$-model in $A- V(B)$ (by collecting only components hitting $\ell$ sets of $\cZ$).
Since
\[\abs{T\cup V(A\cap B)}\le f^1_{H, \ell}(k)+k\ell^2\le f_{H, \ell}(k),\]
we have a required set.
\end{itemize}

Now, we assume that $\ell'\ge 1$. 
Let $\cZ'$ be the multiset of all sets $X$ in $\cZ$ where $X\subseteq V(A)$.
We follow a similar procedure in Theorem~\ref{thm:mainpure}.
We divide cases depending on whether $\cc(H)\ge \ell$ or not.

\begin{itemize}
\item (Case 2-1. $\cc(H)\ge \ell$.) \\
Since $B- V(A)$ contains a $(\cZ\setminus A,k, \ell')$-rooted grid model of order $k\ell(14h+2)+1$, 
by Lemma~\ref{lem:colorfulgrid},
$B- V(A)$ contains $k$ pairwise vertex-disjoint $(\ell'\cdot \cG_{14h}, \cZ\setminus A, \ell')$-models. 

We apply Theorem~\ref{thm:mainpure} to the instance $(A- V(B), \cZ', \ell-\ell', k)$.
Then $A-V(B)$ contains either $k$ pairwise vertex-disjoint pure $(H, \cZ', \ell-\ell')$-models, 
or a pure $(H, \cZ', \ell-\ell')$-deletion set $T$ of size at most $f^1_{H, \ell-\ell'}(k)$.
In the former case, we can obtain $k$ pairwise vertex-disjoint $(H, \cZ, \ell)$-models in $G$
using the $k$ pairwise vertex-disjoint $(\ell'\cdot \cG_{14h}, \cZ\setminus A, \ell')$-models in $B- V(A)$. 
For the latter case, 
$G- \left(T\cup V(A\cap B)\right)$
 has no $(H, \cZ,\ell)$-models where $\abs{T\cup V(A\cap B)}\le f^1_{H, \ell-\ell'}(k) + k\ell^2\le f_{H, \ell}(k)$.

\item (Case 2-2. $\cc(H)= \ell-1$.) \\
If $\ell'\ge 2$, then 
we can obtain the same result
as in Case 2-1, 
because there are in fact
$k$ pairwise vertex-disjoint $((\ell'-1)\cdot \cG_{14h}, \cZ\setminus A, \ell')$-models in $B- V(A)$ by Lemma~\ref{lem:colorfulgrid}.
We may assume that $\ell'=1$.
Since $\cc(H)\ge 2$, we have $\ell\ge 3$.

Now, for each subgraph $H'$ of $H$ induced by its $\ell-2$ connected components, 
we apply Theorem~\ref{thm:mainpure} to the instance $(A- V(B), \cZ', \ell-1, k)$.
Then we deduce that $A-V(B)$ contains either $k$ pairwise vertex-disjoint pure $(H', \cZ', \ell-1)$-models, 
or a pure $(H', \cZ', \ell-1)$-deletion set $T_{H'}$ of size at most $f^1_{H, \ell-1}(k)$.
If $A-V(B)$ contains $k$ pairwise vertex-disjoint pure $(H', \cZ', \ell-1)$-models for some subgraph $H'$ of $H$ induced by its $\ell-2$ connected components, 
then we can complete them using $(\cG_{14h}, \cZ\setminus As, 1)$-models in $B- V(A)$.
Therefore, we may assume that 
there is a pure $(H', \cZ', \ell-1)$-deletion set $T_{H'}$ of size at most $f^1_{H, \ell-1}(k)$ in $A- V(B)$ for all such subgraphs $H'$.
Let $T$ be the union of all deletion sets $T_{H'}$. Note that $\abs{T}\le (\ell-1)f^1_{H, \ell-1}(k)$. 

We claim that $G- \left(T\cup V(A\cap B)\right)$
 has no $(H, \cZ,\ell)$-models.
 Suppose that $G- \left(T\cup V(A\cap B)\right)$ has an $(H, \cZ, \ell)$-models $F$.
 Since $\ell'=1$,
there are $\ell-2$ connected components of $F\cap (A- V(B))$ that meet $\ell-1$ sets of $\cZ'$,
which contains a pure $(H'', \cZ', \ell-1)$-model for some 
subgraph $H''$ of $H$ induced by $\ell-2$ connected components of $H$.
Since $T$ meets all such models, we conclude that $G- \left(T\cup V(A\cap B)\right)$ has no $(H, \cZ,\ell)$-models. 
Since 
\[\abs{T\cup V(A\cap B)}\le (\ell-1)f^1_{H, \ell-1}(k) + k\ell^2\le f_{H, \ell}(k), \]
we have a required $(H, \cZ, \ell)$-deletion set.
\end{itemize}
We conclude that $G$ contains either
$k$ pairwise vertex-disjoint $(H, \cZ, \ell)$-models in $G$, or
an $(H, \cZ, \ell)$-deletion set of size at most $f(k,\ell,h)$.
 \end{proof}

\section{Examples when $H$ has at most $\ell-2$ connected components}\label{sec:counterex}

In this section, we prove that if the number of connected components of $H$ is at most $\ell-2$, then the class of $(H, \cZ, \ell)$-models does not have the Erd\H{o}s-P\'osa property.

\begin{proposition}\label{prop:negative}
Let $\ell$ be a positive integer and $H$ be a non-empty planar graph with at most $\ell-2$ connected components. 
Then the class of $(H, \cZ, \ell)$-models does not have the Erd\H{o}s-P\'osa property.
\end{proposition}

 \begin{figure}[t]
\centerline{\includegraphics{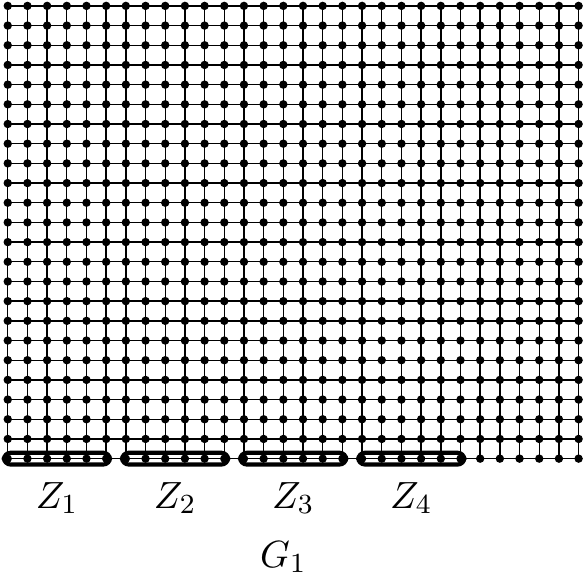} \includegraphics{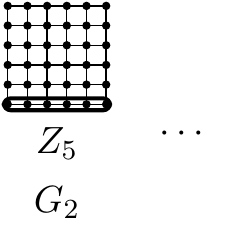} \includegraphics{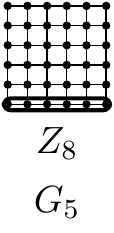}}
  \caption{The construction for showing that $(H, \cZ, \ell)$-models have no Erd\H{o}s-P\'osa property when 
  the number of connected components of $H$ is at most $\ell-2$. This is an example when $\ell=8$ and $\cc(H)=5$.}
  \label{fig:counterex2}
\end{figure}

 \begin{proof}
We show that for every positive integer $x$, there is a rooted graph $(G, \cZ)$ satisfying that
\begin{itemize}
\item $G$ has one $(H, \cZ, \ell)$-model, but  no two vertex-disjoint $(H, \cZ, \ell)$-models,
\item for every vertex subset $S$ of size at most $x$, $G- S$ has an $(H, \cZ, \ell)$-model.
\end{itemize}
It implies that the function for the Erd\H{o}s-P\'osa property does not exist for $k=2$.

Let $H_1, \ldots, H_t$ be the connected components of $H$. Note that $\ell-t\ge 2$ by the assumption.
Let $G$ be the disjoint union of $\cG_{(\ell-t+2)n}$, say $G_1$, and $t-1$ copies of $\cG_{n}$, say $G_2, \ldots, G_t$. 
For each $1\le k\le t$, we denote by $v^k_{i,j}$ for the copy of $v_{i,j}$ in $G_k$. 
For each $1\le j\le \ell-t+1$, let $Z_j:=\{v^1_{1,i}:n(j-1)+1\le i\le nj \}$. 
Also, for each $\ell-t+2\le j\le \ell$, 
let $Z_{j}:=\{v^{j-\ell+t}_{1,i}:1\le i\le n\}$.
Since $\ell-t+1\ge 3$, $G_1$ contains at least $3$ sets of $\cZ$.
We will determine $n$ later.
We depict this construction in Figure~\ref{fig:counterex2}.
 
It is clear that if $n\ge 14h$, then $G$ contains one $(H, \cZ, \ell)$-model
because each $G_i$ contains a $\cG_{14h}$ subgraph, and thus contains a $H_i$-model, and we can take a path from each set of $\cZ$ to the constructed connected component of $H$ in $G_i$.
We observe that there are no two $(H, \cZ, \ell)$-models.
Let $F_1, F_2$ be $(H, \cZ, \ell)$-models.
For each $F_i$, since $H$ consists of exactly $t$ connected components, 
each $G_i$ should contain the image of one connected component of $H$.
Especially, $F_i\cap G_1$ is the image of one connected component of $H$ which meets all sets of $\cZ$ contained in $G_1$.
However, since $F_1\cap G_1$ and $F_2\cap G_1$ are connected, they should intersect.

Now we claim that if $n\ge (14h+x+1)(x+1)+x+1$ and $S$ is a vertex set of size at most $x$ in $G$, 
then $G- S$ contains an $(H, \cZ, \ell)$-model.
Suppose that $n\ge (14h+x+1)(x+1)+x+1$ and $S$ is a vertex set of size at most $x$ in $G$.

Let us fix $2\le j\le t$. Then there exists $1\le p_j\le n-(x+1)$ and $0\le q_j\le n-(x+1)$ such that $v^{j}_{p_j+a, q_j+b}\notin S$ for all $1\le a,b\le 14h+x+1$.
Let $F_j$ be the subgraph of $G_j$ induced by the vertex set $\{v^{j}_{p_j+a, q_j+b}:1\le a,b\le x+1\}$.
Clearly there are $x+1$ pairwise vertex-disjoint paths from $Z_{j+(\ell-t)}$ to $\{v^j_{p_j+1, q_j+b}:1\le b\le x+1\}$ and 
there is at least one path that does not meet $S$.
Since $F_j$ contains an $H_j$-model, $G_j- S$ contains an $H_j$-model having a vertex of $Z_{j+(\ell-t)}$.

Similarly, in $G_1$, 
there exists $1\le p_1\le n-(x+1)$ and $(\ell-t+1)n\le q_1\le (\ell-t+1)n+n-(x+1)$ such that $v^{j}_{p_j+a, q_j+b}\notin S$ for all $1\le a,b\le 14h+x+1$.
Let $F_1$ be the subgraph of $G_1$ induced by the vertex set $\{v^{1}_{p_1+a, q_1+b}:1\le a,b\le x+1\}$.
It is not hard to observe that
there are $x+1$ pairwise vertex-disjoint paths from $\{v^1_{p_1+a, q_1+1}:1\le a\le x+1\}$ to each $Z_i$ for $1\le i\le \ell-t+1$. 
Therefore, $G_1- S$ contains an $H_1$-model having a vertex of $Z_i$ for $1\le i\le \ell-t+1$.
We conclude that $G- S$ has an $(H, \cZ, \ell)$-model, as required.
\end{proof}

\section{Conclusion}

In this paper, we show that the class of $(H, \cZ, \ell)$-models has the \EP\ property if and only if $H$ is planar and $\cc(H)\ge \ell-1$.
Among all the interesting results on the \EP\ property, 
some objects that intersect $\ell$ sets among given vertex sets have not much studied.
For our result, we do not restrict an $H$-model to a minimal subgraph containing $H$-minor. 
What if we consider minimal $H$-models or $H$-subdivisions? Then it may be difficult to have such a nice characterization.
As a first step to study such families, we pose one specific problem:
\begin{itemize}
	\item Does the set of cycles intersecting at least two sets among given sets $Z_1, Z_2, \ldots, Z_m$ in a graph $G$ have the \EP\ property, with a bound on a deletion set that does not depend on $m$?
\end{itemize}	
Huynh, Joos, and Wollan~\cite{HuyneJW2017} showed that $(S_1, S_2)$-cycles have the \EP\ property.
Thus, if this is true, then it generalizes the result for $(S_1, S_2)$-cycles.

\end{document}